\documentclass[11pt]{amsart}
\usepackage{tikz-cd}
\usepackage{enumitem}
\setlist[description]{leftmargin=\parindent,labelindent=\parindent}
\usetikzlibrary{matrix,arrows,decorations.pathmorphing}
\usepackage{amssymb}
\usepackage{amsmath,amscd}
\usepackage{mathrsfs}
\usepackage{url}
\usepackage{cite}
\usepackage{fullpage}
\usepackage{hyperref}
\usepackage{verbatim}
\usepackage{comment}
\usepackage{fancyvrb}
\usepackage{fvextra}
\usepackage{caption}
\usepackage{ytableau}
\usepackage{stmaryrd}
\usepackage{appendix}

\newtheorem{thm}{Theorem}
\newtheorem{prop}[thm]{Proposition}
\newtheorem{lem}[thm]{Lemma}
\newtheorem{cor}[thm]{Corollary}
\newtheorem{conj}[thm]{Conjecture}

\theoremstyle{definition}

\newtheorem{rem}[thm]{Remark}

\newcommand{\Z}{\mathbb{Z}}

\newcommand{\cc}{\mathbb{C}}

\newcommand{\qq}{\mathbb{Q}}

\newcommand{\M}{\mathcal{M}}
\newcommand{\p}{\mathbb{P}}
\newcommand{\pp}{\mathbb{P}}

\newcommand{\F}{\mathcal{F}}

\DeclareMathOperator{\quadr}{quad}

\renewcommand{\O}{\mathcal{O}}

\renewcommand{\tilde}{\widetilde}

\DeclareMathOperator{\ctp}{ctp}

\DeclareMathOperator{\GL}{GL}
\DeclareMathOperator{\SL}{SL}

\DeclareMathOperator{\Sym}{Sym}

\DeclareMathOperator{\PGL}{PGL}
\DeclareMathOperator{\PSL}{PSL}
\DeclareMathOperator{\BSL}{BSL}
\DeclareMathOperator{\BSO}{BSO}

\DeclareMathOperator{\BGL}{BGL}

\renewcommand{\gg}{\mathbb{G}}

\usepackage{wrapfig}

\renewcommand{\bar}{\overline}
\newcommand{\Mb}{\overline{\M}}

\title{The Chow ring of the moduli space of degree $2$ quasi-polarized K3 surfaces}
\author{Samir Canning}

\address{Department of Mathematics, ETH Z\"urich}
\email {samir.canning@math.ethz.ch}
\author{Dragos Oprea}
\address{Department of Mathematics, University of California, San Diego}
\email {doprea@math.ucsd.edu}
\author{Rahul Pandharipande}
\address{Department of Mathematics, ETH Z\"urich}
\email {rahul@math.ethz.ch}
\date{June 2024}

\begin{document}

\baselineskip=16.5pt

\begin{abstract} We study the Chow ring with rational coefficients of the moduli space $\mathcal F_{2}$ of quasi-polarized $K3$ surfaces of degree $2$. We find generators, relations, and
calculate the
Chow Betti numbers.
The highest nonvanishing Chow group is 
$\mathsf A^{17}(\mathcal F_2)\cong {\mathbb{Q}}$. 
We prove
that the Chow ring consists of tautological classes and
is isomorphic to the even
cohomology. 
The Chow ring is not generated
by divisors and does not
satisfy duality with  respect to
the pairing into $\mathsf A^{17}(\mathcal F_2)$. 
The kernel of the pairing is
a 1-dimensional subspace of
$\mathsf{A}^{9}(\mathcal F_2)$ which we calculate explicitly.
In the appendix, we revisit Kirwan-Lee's calculation of the Poincar\'e polynomial of $\F_2$.

\end{abstract}
\maketitle
\setcounter{tocdepth}{1}
\tableofcontents

\section{Introduction} 

\subsection{Main results} Let $\mathcal F_{2\ell}$ denote the moduli space of quasi-polarized $K3$ surfaces of degree $2\ell$. The space $\mathcal F_{2\ell}$ is a nonsingular Deligne-Mumford stack of dimension $19$. We consider the Chow ring $\mathsf A^*(\mathcal F_{2\ell})$, which will be taken with $\mathbb Q$-coefficients throughout the paper. The Chow of ring admits a tautological subring $$\mathsf R^{*}(\mathcal F_{2\ell})\subset \mathsf A^{*}(\mathcal F_{2\ell})\, ,$$ which was defined in \cite {MOP} and will be reviewed in Section \ref{tring}.

We focus here on 
the moduli space $\mathcal F_{2}$ of $K3$ surfaces of degree $2\ell=2$. The generic
$K3$ surface 
$(S,L) \in \F_{2}$ 
is a double cover  $$\epsilon: S \rightarrow {\mathbb{P}^2}$$ branched
along a nonsingular sextic curve with quasi-polarization 
$L=\epsilon^*{\mathcal{O}}_ {\mathbb{P}^2}(1)$. The geometry of 
$\mathcal F_{2}$ is
therefore closely related to the classical geometry of sextic plane curves. 

\begin{thm}\label{main}
The following results hold for the Chow ring $\mathsf A^*(\mathcal F_2)$:
\begin{itemize}
\item [\textnormal{(i)}] The  Chow ring is tautological, $$\mathsf R^{*}(\mathcal F_2)=\mathsf A^{*}(\mathcal F_2)\, ,$$ and is generated by 4 elements of degrees $1, 1, 2, 3$ respectively. 
\item [\textnormal{(ii)}] $\mathsf R^{17}(\mathcal F_2)=\mathbb Q$ and $\mathsf R^{18}(\mathcal F_2)=\mathsf R^{19}(\mathcal F_2)=0$.  
\item [\textnormal{(iii)}] The dimensions are given by \begin{align*}\sum_{k=0}^{19} t^k \dim \mathsf R^k(\mathcal F_2)=1 + &2t + 3t^2 + 5t^3 + 6t^4 + 8t^5 + 10t^6 + 12t^7 + 13t^8 + \\&+ 14t^9 + 12t^{10} +10t^{11} + 8t^{12} + 6t^{13} + 5t^{14} + 3t^{15} + 2t^{16} + t^{17}.\end{align*}
\item[\textnormal{(iv)}] The cycle class map is 
an isomorphism onto the
even cohomology: 
$$\forall k\geq 0\,, \ \ \ \mathsf R^k(\mathcal F_2) 
\cong
H^{2k}(\mathcal F_2)\, .$$
\end{itemize} 
\end{thm}

Theorem \ref{main} is the first
complete Chow calculation for
the moduli spaces ${\mathcal F}_{2\ell}$ of
quasi-polarized $K3$ surfaces. There are several immediate connections and consequences:

\vspace{5pt}
\noindent $\bullet$ The generators (i) of $\mathsf R^{*}(\mathcal{F}_2)$
are {\em not} all divisor classes. Indeed, the Chow Betti numbers given in (iii) grow too quickly to be generated by the two divisor classes.

\vspace{5pt}
\noindent $\bullet$ In the proof of Theorem \ref{main}, we determine the structure of the Chow ring of $\mathcal F_2$: there are $14$ relations{\footnote{The
relations are recorded in Remark \ref{nonreducedremark} (3 relations), Remark \ref{quadremark} (3 relations), Remark \ref{twomore} (2 relations), Remark \ref{r15} and Remark \ref{r16} (together 6 relations). In these 14 relations, the $4$ generators are denoted by $H$ (degree 1), $\mathsf{Ell}$ (degree 1), $c_2$ (degree $2$), $c_3$ (degree 3).}} 
between the 4 generators (i). 
\vspace{5pt}

\noindent $\bullet$ The socle and
vanishing results of (ii)
were proposed earlier 
as analogues of Faber's conjectures \cite{Faber} for the tautological ring of the moduli space of curves $\mathcal M_g$: $$\mathsf R^{>g-2}(\mathcal M_g)=0\, , \quad \mathsf R^{g-2}(\mathcal M_g)=\mathbb Q\, .$$

\begin{conj}[Oprea--Pandharipande (2015)]\label{c1}
Let $\Gamma$ be the Picard lattice for a K3 surface such that $d= 20 - \text{\em rank}(\Gamma)>3$.
For the moduli space $\mathcal F_{\Gamma}$ of $\Gamma$-polarized $K3$ surfaces, we have
$$\mathsf R^{d-2}(\mathcal F_{\Gamma})\cong \mathbb Q\ \ \ \text{and} \ \ \
\mathsf R^{d-1}(\mathcal F_\Gamma)=\mathsf R^{d}(\mathcal F_\Gamma)=0\, .$$
\end{conj}

\noindent In cohomology with ${\mathbb{Q}}$-coefficients, the vanishing part
of Conjecture \ref{c1} is 
established in \cite{P}. For the hyperbolic lattice $U$, the moduli space $\mathcal F_U$ corresponds to elliptic $K3$
surfaces with section. The socle and 
Chow vanishing  properties of Conjecture \ref{c1} for
${\mathcal{F}}_U$ are established in \cite{CK}.
The  moduli space $\mathcal{F}_2$
is the first rank 1 case
where Conjecture \ref{c1} is proven.

\vspace{5pt}
\noindent $\bullet$ The Chow Betti number calculation (iii) shows that the pairing into
$\mathsf R^{17}(\mathcal{F}_2)$ is not 
perfect because the middle dimensions are not equal:
$$\dim \mathsf R^8(\mathcal{F}_2) \neq \dim \mathsf R^9(\mathcal{F}_2)\, .$$
In fact, the kernel
of the pairing $$\mathsf{R}^8(\mathcal F_2)\times \mathsf{R}^9(\mathcal F_2)\to \mathsf R^{17}(\mathcal{F}_2)$$ has rank $1$.   A formula for the class
in $\mathsf{R}^9(\mathcal F_2)$
which spans the
Gorenstein kernel is given in Remark \ref{hc2}. Finding a geometric
interpretation of the kernel
class is a very interesting open direction.
For $k$ not equal to 8 or 9, the pairing 
$$\mathsf{R}^k(\mathcal F_2)\times \mathsf{R}^{17-k}(\mathcal F_2)\to \mathsf R^{17}(\mathcal{F}_2)$$ is perfect. 
The results about the pairing are proven using  our explicit presentation for $\mathsf{R}^*(\F_2)$.

\vspace{5pt}
\noindent $\bullet$ A construction of the moduli space $\mathcal F_2$ as an open subset of 
a weighted blowup of the moduli space of sextic curves in $\mathbb P^2$ is given in \cite {S}, see \cite {La, Lo} for a summary. A partial desingularization of the full GIT compactification of the space of sextics is obtained in \cite {KL1} and requires three further blowups.  
Using all four blowups, the cohomology of $\mathcal F_2$ was studied in \cite {KL1, KL2}. A main result of \cite {KL2} is the calculation of the Poincar\'e polynomial\footnote{The value of the Poincar\'e polynomial given in \cite {KL2} is 
\begin{multline*} 1 + 2q^2 + 3q^4 + 5q^6 + 6q^8 + 8q^{10} + 10q^{12} + 12q^{14} + 13q^{16} + 14q^{18} + 12q^{20} \\+ 10q^{22}  + 8q^{24} + 6q^{26} + q^{27} + 5q^{28} + 3q^{30} + q^{31} + 2q^{32} + q^{33} + 3q^{35}\, ,\end{multline*} which differs from the statement above by $q^{33}+q^{34}$. We will explain the necessary correction in Appendix \ref{appendix}.}
\begin{multline*}\sum_{k=0}^{38} q^k \dim H^k(\mathcal F_2)=1 + 2q^2 + 3q^4 + 5q^6 + 6q^8 + 8q^{10} + 10q^{12} + 12q^{14} + 13q^{16} + 14q^{18} + 12q^{20}  \\ + 10q^{22}+ 8q^{24} + 6q^{26} + q^{27} + 5q^{28} + 3q^{30} + q^{31} + 2q^{32} + 2q^{33} + q^{34} + 3q^{35} .\end{multline*} While there are odd cohomology classes, the dimensions of the even cohomology of $\mathcal F_{2}$
match the Chow Betti numbers (iii)
as required by (iv).

\vspace{5pt}
\noindent $\bullet$ By a result of \cite{Ber}, the even cohomology $H^{2k}(\mathcal{F}_{2\ell})$ for $k\leq 4$ is tautological for all $\ell\geq 1$. Isomorphism (iv) is
a much stronger property which holds for the moduli space $\mathcal{F}_2$.

\subsection{Plan of the paper}
Definitions and basic results
related to tautological classes on moduli spaces
of $K3$ surfaces are reviewed in Section \ref{tring}.
Our approach to the geometry of $\mathcal{F}_2$
relies upon Shah's blowup construction \cite{S}
which is discussed in Section \ref{Shah}. Part (i) of Theorem
\ref{main} is proven in Section \ref{ade6}.

Shah describes $\F_2$ as an open subset of a weighted blowup of
the space of sextic plane curves.
The heart of our Chow calculation for $\mathcal{F}_2$
is presented in Sections \ref{ML}--\ref{relctp}, where
relations obtained from the removal of various loci are determined. 
Parts (ii) and (iii) of Theorem \ref{main} are proven
in Section \ref{relctp}. 

Part (iv) of Theorem \ref{main}, the isomorphism of
the cycle class map onto the even cohomology, is proven in Section \ref{cymap}.

The Kirwan-Lee calculation of the Poincar\'e polynomial of $\mathcal{F}_2$ 
is discussed carefully in Appendix \ref{appendix}. In particular, we explain  how to correct the calculations in \cite {KL1, KL2}. 

\subsection{Future directions}
While complete Chow calculations for the moduli spaces
$\mathcal{F}_{2\ell}$ will likely become
intractable for large $\ell$, the study of 
$\mathcal{F}_4$ should be possible as there
is a parallel (though more complicated) construction
starting from the moduli of quartic surfaces,
see \cite{LOG}.

Another direction of study is to find structure in
the tautological ring $\mathsf{R}^*(\mathcal{F}_{2\ell})$
beyond Conjecture~\ref{c1}. The parallel direction
in the study of the moduli space of curves has led
to the surprising discovery of uniform sets of
tautological relations, see \cite{utah} for a survey.
Whether any analogues of the Faber-Zagier and Pixton relations
hold for the moduli of $K3$ surfaces is an interesting question.

Finding algebraic cycle classes
on $\mathcal{F}_{2\ell}$ which are
non-tautological in cohomology is another open direction. 
Since such classes for the moduli of curves
and abelian varieties can be constructed using the geometry of Hurwitz
covers of higher genus curves \cite{COP2,GP,Z}, a simple idea for
$K3$ surfaces is the following.
Let 
$\mathcal{B}_g\subset {\mathcal{F}}_{2\ell}$
be the closure of the locus 
of $K3$ surfaces 
for which there exists a nonsingular linear section (of genus $\ell+1$) admitting a degree 2 map to a genus $g\geq 1$ curve.
A reasonable expectation is that the image in cohomology of 
the algebraic cycle class 
$$[\mathcal{B}_g]\in \mathsf{A}^*(\mathcal{F}_{2\ell})$$
is non-tautological for appropriate choices of 
$g$ and $\ell$. 

\subsection*{Acknowledgments}
We thank Dan Abramovich, Veronica Arena, Adrian Clingher, Carel Faber, Frances Kirwan, Bochao Kong, Hannah Larson, Zhiyuan Li, Alina Marian, Davesh Maulik, Stephen Obinna, Dan Petersen, and Burt Totaro for many related conversations
over the years. We are very grateful to the referee for several comments and corrections.

An earlier attempt to 
calculate the Chow ring of
${\mathcal{F}}_2$ with
Qizheng Yin and Fei Si did not
succeed (due to the geometric and
computational complexity of the
approach). A
different path 
is taken
here, leading to several simplifications. We thank Qizheng Yin and Fei Si
for the previous collaboration.

S.C. was supported by a Hermann-Weyl-Instructorship from the Forschungsinstitut f\"ur Mathematik at ETH Z\"urich.
D.O. was supported by  NSF-DMS 1802228.
R.P. was supported by
SNF-200020-182181, SNF-200020-219369,
ERC-2017-AdG-786580-MACI, and SwissMAP. 

This project has
received funding from the European Research Council (ERC) under the Euro-
pean Union Horizon 2020 research and innovation program (grant agreement No. 786580).

\section{Tautological classes} \label{tring} We review here the definition of  the tautological rings
of the moduli spaces $\mathcal F_{2\ell}$. 
Consider the universal $K3$ surface and quasi-polarization $$\pi:\mathcal S\to \mathcal F_{2\ell}\ ,\ \ \  \mathcal L\to \mathcal S\, .$$ 
The most basic tautological classes are:

\vspace{5pt}
\noindent $\bullet$ {\it Hodge classes.} The dual Hodge bundle is defined as the pushforward $$\mathbb E^{\vee}={\mathbf R}^2 \pi_{*} {\mathcal O}_{\mathcal S}\to \mathcal F_{2\ell}\, .$$ Let $\lambda \in \mathsf A^1 ({\mathcal F}_{2\ell})$ denote the first Chern class 
$$\lambda = c_1(\mathbb E)=- c_1 \left ({\mathbf R}^2 \pi_{*} {\mathcal O}_{\mathcal S} \right )\, .$$ The ring $\Lambda^{*}(\F_{2\ell})\subset \mathsf A^{*} ({\mathcal F}_{2\ell})$ generated by powers of $\lambda$ was studied in \cite{kvg}. It is shown there that $$\Lambda^{*}(\mathcal F_{2\ell})=\mathbb Q[\lambda]/(\lambda^{18})\, .$$  
The Chern classes of the tangent bundle of $\mathcal F_{2\ell}$ belong to the ring $\Lambda^{*}.$ In fact, by \cite{kvg} (see the paragraph following Proposition 3.2) we have $$\mathrm{ch}(T\mathcal F_{2\ell})=-1+21 e^{-\lambda}-e^{-2\lambda}\, .$$

\vspace{5pt}
\noindent $\bullet$ {\it Noether-Lefschetz classes.} 
Let $\Gamma$ be an even lattice of signature $(1, r-1)$ for an integer $r\leq 19$. 
Consider the moduli space $\mathcal F_{\Gamma}$ parametrizing $\Gamma$-polarized $K3$ surfaces: $$\iota:\Gamma\to \text{Pic}(S)\, ,$$ with the image of $\iota$ containing a quasi-polarization of the surface $S$. Upon fixing a primitive $v\in \Gamma$ with $v^2=2\ell$ mapping to the quasipolarization, there is a forgetful map $$\mathcal F_{\Gamma}\to \mathcal F_{2\ell}\, ,$$ whose image determines a 
Noether-Lefschetz
cycle 
in $\mathcal F_{2\ell}$.  Define 
$$\mathsf {NL}^{*}(\mathcal F_{2\ell})\subset \mathsf A^{*}(\mathcal F_{2\ell})$$ to be the $\mathbb{Q}$-subalgebra  generated by all Noether-Lefschetz cycle
classes. A more extensive discussion of the Noether-Lefschetz cycles can be found for instance in \cite {MP}.

\vspace{5pt}
\noindent $\bullet$
{\it Kappa classes.} 
Let $T_{\pi}^{\text{rel}}$ be the
relative tangent bundle of the universal
surface $\pi:\mathcal S\to \mathcal F_{2\ell}$. The first
Chern class is related to the Hodge class,
$$c_1 \left (T^{\text{rel}}_{\pi}\right ) = - \pi^{*} \lambda\, .$$ Define 
$t = c_2 \left (T_{\pi}^{\text{rel}}\right )$. By \cite [Proposition 3.1]{kvg}, the pushforwards $\kappa_{0,n}=\pi_{*} (t^{n})$ are contained in $\Lambda^{*}$. 
More general kappa classes can be defined by including the class of the quasi-polarization
$c_1(\mathcal L)$ and considering the pushforwards 
$$\kappa_{m,n} = \pi_{*} \left (c_1(\mathcal L)^m\cdot t^n \right)\, .$$ These classes depend upon the normalization of $\mathcal L$ by line bundles pulled back from $\mathcal F_{2\ell}$. By defining canonical normalizations for {\em admissible} $\mathcal{L}$, the ambiguity can be removed, see \cite {PY}.

\vspace{5pt}  

More generally, we define enriched kappa classes over $\mathcal F_{\Gamma}$ and consider their pushforwards to $\mathcal F_{2\ell}$. After picking a basis $\mathsf{B}$ for $\Gamma$, we obtain line bundles $\mathcal L_1, \ldots, \mathcal L_r\to \mathcal F_{\Gamma}$, and we set
\begin{equation} \label{genkap}
\kappa^{\Gamma, \mathsf B}_{a_1, \ldots, a_r, b}=\pi_{*} \left(c_1(\mathcal L_1)^{a_1} \cdot \ldots \cdot c_1(\mathcal L_r)^{a_r} \cdot c_2(T^{\textrm{rel}}_{\pi})^{b}\right)\, .
\end{equation}
The tautological ring of the 
moduli space ${\mathcal F_{2\ell}}$, $$ \mathsf R^{*}(\mathcal F_{2\ell}) \subset \mathsf A^{*}(\mathcal F_{2\ell})\, ,$$ is defined in \cite{MOP}
to be the $\mathbb{Q}$-subalgebra  generated by the pushforwards of all enriched kappa classes \eqref{genkap}
 for all possible $\Gamma$. 
 
 By definition, $\mathsf R^{*}(\mathcal F_{2\ell}) \supset \mathsf {NL}^{*}(\mathcal F_{2\ell})$.
 The following
 isomorphism was proven in \cite{PY}.

\begin{thm}[Pandharipande--Yin (2020)] For all $\ell\geq 1$, we have
    $\mathsf R^{*}(\mathcal F_{2\ell})=\mathsf {NL}^{*}(\mathcal F_{2\ell})$.
\end{thm}

Stronger results hold for divisor classes.
In codimension $1$, the isomorphism $$\text{Pic} (\mathcal F_{2\ell})_{\mathbb Q}=\mathsf {NL}^1(\mathcal F_{2\ell})$$ was conjectured in \cite {MP} and proven in \cite {Ber}. Combined with \cite {B}, this isomorphism determines the Picard rank. The integral Picard group has recently been considered in \cite {LFV}. Furthermore, bases for the rational Picard group for small $\ell$ are given in \cite {OG} and \cite{GLT}. 

For ${\mathcal F_2}$, consider the divisor $\mathsf{Ell}$ of elliptic $K3$ surfaces with a section and the divisor $\mathsf {Sing}$ of $K3$ surfaces for which the quasi-polarization fails to be ample. These are Noether-Lefschetz divisors corresponding to the lattices $$\begin{bmatrix} -2 & 1 \\ 1 & 0\end{bmatrix} \ \ \ \text{and} \ \ \ \begin{bmatrix} 2 & 0 \\ 0 & -2\end{bmatrix} \, ,$$
respectively. 
By \cite{OG}, we have $$\text{Pic}(\mathcal F_2)_{\mathbb Q}=\langle \left[\mathsf{Ell}\right], \left[\mathsf {Sing}\right]\rangle$$ which is consistent with part (iii)
of Theorem \ref{main}. In fact, the arguments of \cite {OG} or alternatively \cite [Proposition 1]{MOP} imply that $\lambda$ is a combination of $\left[\mathsf{Ell}\right]$ and $\left[\mathsf{Sing}\right]$ with nonzero coefficients. Thus, we also have  \begin{equation}\label{ogrady}\text{Pic}(\mathcal F_2)_{\mathbb Q}=\langle \lambda, \left[\mathsf{Ell}\right]\rangle\, .\end{equation}

\section{Shah's construction}\label{Shah}

\subsection{Overview} We review here the construction of 
$\mathcal F_2$  described in \cite {S}. We start with the moduli space of sextics with suitably restricted singularities. 
The moduli space $\mathcal F_2$ is
then obtained as an open subset of a weighted blowup 
of the locus of triple conics.
Using the construction, we will prove part (i) of Theorem \ref{main}. 

\subsection {ADE sextics} \label{ade6}
By a result of \cite {Ma}, a quasi-polarized $K3$ surface $(S, L)$ of degree $2$ 
must take one of the following
two geometric forms:
\begin{itemize}
\item [(a)] $S$  is the resolution of a double cover $\epsilon:\widehat S\to \mathbb P^2$ branched along a sextic curve $C$ with
$L=\epsilon^*{\mathcal{O}}_ {\mathbb{P}^2}(1)$.
 For $S$ to be a nonsingular $K3$ surface, $C$ must have ADE singularities \cite {BHPV}. 
\item [(b)] $S$ is  an elliptic fibration $S\to \mathbb P^1$ with fiber class $f$, section $\sigma$, and $L=\sigma+2f$. 
\end{itemize}
Surfaces $(S,L)$ of form (b) constitute the  Noether-Lefschetz divisor $\mathsf{Ell}$ in $\mathcal F_2$.
The complement $\mathcal F_2\smallsetminus \mathsf{Ell}$ is a gerbe banded by a finite group over the moduli space of plane sextics with ADE singularities.

We will construct the moduli space of
plane sextics with ADE singularities 
as a quotient stack.
Let $V=\mathbb C^3$, and let $$X=\mathbb P(\Sym^6 V^*)=\mathbb P^{27}$$ 
be the space of all sextic curves in $\mathbb P(V)$. A sextic $C\subset \mathbb P^2$ has at worst ADE singularities if and only if the following three conditions are satisfied simultaneously 
\begin{itemize}
\item [(i)] $C$ is reduced,
\item [(ii)] $C$ does not contain a consecutive triple point,
\item [(iii)] $C$ does not contain a quadruple point.
\end{itemize}
By definition, a consecutive triple point $p\in C$ is a triple point such that the strict transform of $C$ in the blowup of $\mathbb P^2$ at $p$ continues to admit a triple point. In local coordinates $(x,y)$ around $p$, the equation for $C$ lies in the ideal $(x,y^2)^3$.
We write $\mathsf{NR}$, $\mathsf {CTP}$, $\mathsf {QP}$ for the three loci failing conditions (i)-(iii) respectively. 
We have a morphism 
\begin{equation}
\label{fff1} \F_2\smallsetminus \mathsf{Ell}\rightarrow \left[\left(X\smallsetminus (\mathsf{NR}\cup \mathsf {CTP}\cup\mathsf {QP})\right)/\PGL(V)\right]\, ,  
\end{equation}
of degree $1/2$ (because of the extra
$\Z_2$ automorphisms of the $K3$ surfaces). We also have a
degree $1/3$ morphism
\begin{equation} 
\label{fff2}
\left[\left(X\smallsetminus (\mathsf{NR}\cup \mathsf {CTP}\cup\mathsf {QP})\right)/\SL(V)\right] \rightarrow
\left[\left(X\smallsetminus (\mathsf{NR}\cup \mathsf {CTP}\cup\mathsf {QP})\right)/\PGL(V)\right]
\,. \end{equation}
induced by the map $\SL(V)\rightarrow \PGL(V)$.
Both \eqref{fff1} and \eqref{fff2}
yield isomorphisms on 
Chow
groups with $\mathbb{Q}$-coefficients via pullback. The group $G=\SL(V)$ will be
geometrically
simpler for us to work with
because $\BSL(V)$
carries a universal vector bundle.

The locus $\mathsf{NR}$ is reducible with components corresponding to the images of the maps
\[
m_r:\p (\Sym^{6-2r} V^*) \times \p (\Sym^r V^*)\rightarrow X
, \quad m_r(f,g)=fg^2\] for $r=1,2,3$. 
Let $\mathsf{ML}\subset \mathsf{NR}$ denote the irreducible component of $\mathsf{NR}$ corresponding to sextics containing a multiple line (the image of $m_1$). The locus $\mathsf{ML}$ has codimension $11$ in $X$. The other components of $\mathsf {NR}$ are the double conic locus (the image of $m_2$) and the double cubic locus (the image of $m_3$). These have codimensions $17$ and $18$ in $X$. We will consider these two components together, thus writing $\mathsf{DC}$ for their union. The quasi-projective locus of cubes of {\em nonsingular} conics will be denoted $\mathsf {TC}$. It is easy to see that the loci $$\mathsf {QP}, \, \mathsf {CTP}, \,\mathsf{TC}\subset X$$ have codimensions $8$, $9$ and $22$ respectively. 

\begin{proof}[Proof of Theorem \ref{main} (i).] Throughout the paper, we identify Chow groups of quotient stacks with the corresponding equivariant Chow groups \cite {EG, Kr}. 

We first consider the Chow classes in \[\mathsf A^*(\F_2\smallsetminus \mathsf{Ell})=\mathsf A^*_G(X\smallsetminus (\mathsf{NR}\cup \mathsf {CTP}\cup\mathsf {QP}))\, .
\] Note that $\mathsf A_G^*=\mathbb Z[c_2, c_3]$,
 where $c_2,c_3$ are the second and third Chern classes of the universal bundle over $\mathrm{B}G$.
The $\mathsf A_G^*$-algebra $\mathsf A_G^*(X)$ is generated by $H$, the equivariant hyperplane class. Since we have a surjection
$$\mathsf A_G^{*}(X)\to \mathsf A^*_G(X\smallsetminus (\mathsf{NR}\cup \mathsf {CTP}\cup\mathsf {QP}))\, ,$$ we see that $\mathsf A^*(\F_2\smallsetminus \mathsf{Ell})$ is generated by the images of $H,$ $c_2,$ $c_3$. By \eqref{ogrady}, $\mathsf {Pic}(\mathcal F_2)$ is generated by $\lambda$ and $\left[\mathsf {Ell}\right]$. The class $H$ is the restriction of a linear combination $a\lambda+b [\mathsf {Ell}]$ to $\mathcal F_2\smallsetminus \mathsf{Ell}$, and therefore
must be a multiple of $\lambda$. Thus, the ring $\mathsf A^{*}(\mathcal F_2\smallsetminus \mathsf {Ell})$ is generated by $\lambda,$ $c_2,$ $c_3.$

Over the locus $\mathcal F_2\smallsetminus \mathsf{Ell}$, the classes $c_2,$ $c_3$ are the Chern classes of a rank $3$ vector bundle $\mathcal V$ which can be explicitly described. Let $$\pi:(\mathcal S, \mathcal L)\to \mathcal F_2$$ be the universal surface and the universal polarization. 
For surfaces in the locus $\mathcal F_2\smallsetminus \mathsf{Ell}$, the quasi-polarization $L\to S$ is globally generated inducing a morphism $S\to \mathbb P(H^0(S, L))$. 
Then $c_2, c_3$ are the Chern classes of the bundle{\footnote{
The third root $\det (\pi_* (\mathcal L))^{-\frac{1}{3}}$ can either be viewed formally for the Chern class calculation or can be viewed as an actual line bundle on the
fiber product of \eqref{fff1} and \eqref{fff2}.}}
$$\mathcal V=\pi_{*}\left(\mathcal L\right) \otimes \det (\pi_{*} \left(\mathcal L)\right)^{-\frac{1}{3}}\, .$$ A Grothendieck-Riemann-Roch calculation explained in shows that $c_2, c_3$ are the restrictions of tautological classes on $\mathcal F_2\smallsetminus \mathsf{Ell}.$ The Grothendieck-Riemann-Roch calculation will be made more precise in Remark \ref{grr} below. 

As  shown in \cite [Theorem 4.1]{CK}, the Chow classes on $\mathsf{Ell}$ are polynomials in (the restrictions of) $\lambda$ and a certain codimension $2$ class. In fact, there is a quotient presentation $$\mathsf {Ell}=[Z/\PSL_2]\to \textrm{BPSL}_2$$ and the codimension $2$ class is obtained by pullback from the base. As we will see below, the stack $\mathcal F_2$ also admits a quotient presentation $$\mathcal F_2=[W/G]\to \textrm{B}G$$ hence the class $c_2$ also makes sense on $\mathcal F_2$ by pullback from the base. Furthermore, there is a natural morphism $$\PSL_2\to G\, , \ \  \ g\mapsto \Sym^2 g$$ compatible with the above maps. Geometrically, this corresponds to the fact that the linear series $|L|$ induces a map from $S$ to a nonsingular conic in $\mathbb P^2$ in the elliptic case. Consequently, the restriction of $c_2$ from $\mathcal F_2$ to $\mathsf{Ell}$ can be chosen to be the degree $2$ generator on $\mathsf {Ell}.$ Thus $\mathsf{A}^{*}(\mathsf {Ell})$ is generated by $\lambda$ and $c_2$. Using the push-pull formula, we conclude that the image of $\mathsf A^*(\mathsf {Ell})$ in $\mathsf{A}^{*}(\mathcal F_2)$ lies in the subring generated by $[\mathsf{Ell}],$ $\lambda,$ $c_2.$

By excision, there is an exact sequence $$\mathsf A^{*-1}(\mathsf{Ell}) \rightarrow \mathsf A^*(\F_2) \rightarrow \mathsf A^*(\F_2\smallsetminus \mathsf{Ell}) \rightarrow 0\, .
$$ 
As a result, $\lambda$, $[\mathsf{Ell}],$ $c_2$ and $c_3$ suffice to generate the ring $\mathsf A^*(\mathcal F_2).$ 
We also conclude that the entire Chow ring is tautological. 
\end{proof}

\begin{rem} The proof of Theorem \ref{main} (i)  shows that we can choose the degree $1$ generators to be $\lambda$ and $\left[\mathsf {Ell}\right].$ 
However, the degree 1 generators $$\alpha_1=H, \quad \alpha_2=\left[\mathsf{Ell}\right]$$ 
will be more convenient for us. A simple calculation using
test curves in $\mathcal F_2$ yields the change of basis equation
$$\lambda = \frac{1}{2} H - [\mathsf{Ell}]\, .$$ Indeed, the identity is correct on any complete curve contained in $\mathsf{Ell}$ since $$[\mathsf {Ell}]^2=-\lambda\cdot [\mathsf {Ell}]\, , \ \,H\cdot [\mathsf {Ell}]=0\, ,$$ see \cite [Lemma 1.2]{OG} for the first intersection.  
For a second test curve, we consider a double cover of $\mathbb P^1\times \mathbb P^2$ branched along a smooth $(2, 6)$ hypersurface. This construction yields a family of sextic double planes, for which $\lambda$, $H$ and $[\mathsf{Ell}]$ evaluate to $1$ (using Riemann-Hurwitz), $2$ and $0$ respectively. 
\end{rem}

\begin{rem} \label{grr} We give the details of the Grothendieck-Riemann-Roch calculation alluded to in the proof above. Recall that $c_2, c_3$ are the Chern classes of the bundle $$\mathcal V=\pi_{*}\left(\mathcal L\right) \otimes \det (\pi_{*} \left(\mathcal L)\right)^{-\frac{1}{3}}\,.$$ By Grothendieck-Riemann-Roch, we have \begin{align*}\text{ch}(\pi_{*} \mathcal L)&=\pi_{*}\left(\exp (c_1(\mathcal L))\cdot \text{Todd}(T_\pi^{\text{rel}})\right)\,.\end{align*} 
We expand
$$\exp (c_1(\mathcal L))=1+c_1(\mathcal L)+\frac{c_1(\mathcal L)^2}{2}+\frac{c_1(\mathcal L)^3}{6}+\frac{c_1(\mathcal L)^4}{24}+\ldots\,,$$  $$\text{Todd}(T_\pi^{\text{rel}})=1-\frac{\lambda}{2}+\frac{\lambda^2+t}{12}-\frac{\lambda t}{24}+\frac{-\lambda^4+4\lambda^2t+3t^2}{720}+\ldots\,, $$ where $t=c_2(T_\pi^{\text{rel}}).$ From here, using the definition of the $\kappa$-classes and the fact that $\pi_* (t)=24$, $\pi_{*} (t^2)=88\lambda^2$ by \cite[Section 3]{kvg}, we immediately obtain   
$$\text{ch}_1(\pi_{*} \mathcal L)=\frac{2\kappa_{3, 0}+\kappa_{1, 1}}{12}-\frac{3\lambda}{2},\quad \text{ch}_2(\pi_{*} \mathcal L)=\frac{\kappa_{4,0}+\kappa_{2,1}}{24}-\frac{\lambda}{24}\cdot (2\kappa_{3, 0}+\kappa_{1,1})+\frac{7\lambda^2}{12}\,.$$
Using $c_1(\mathcal V)=0$, we find $$ c_2(\mathcal V)=-\text{ch}_2(\mathcal V)=-\text{ch}_2(\pi_* (\mathcal L))+\frac{1}{6} \left(\text{ch}_1(\pi_{*} (\mathcal L))\right)^2\,.$$ Substituting, we obtain $$c_2=-\frac{1}{24} \left(\kappa_{4, 0}+\kappa_{2, 1}-\frac{1}{36}(2\kappa_{3, 0}+\kappa_{1, 1})^2+{5\lambda^2}\right)\,.$$ The calculation for $c_3$ is entirely similar, but yields a more complicated expression: \begin{align*}c_3=\frac{1}{360}(3\kappa_{1, 2}+10\kappa_{3,1}+6\kappa_{5,0})
&+\frac{1}{23328}(2\kappa_{3,0}+\kappa_{1,1})^3-\frac{1}{432}(2\kappa_{3,0}+\kappa_{1,1})(\kappa_{2,1}+\kappa_{4,0})\\&-\frac{1}{1080}(23\kappa_{1,1}+40\kappa_{3,0})\lambda^2.\end{align*} Both Chow classes $c_2$, $c_3$ are well-defined on the moduli space $\mathcal F_2$ (and not only on the related quotients \eqref{fff1}, \eqref{fff2} which appear in the proof of Theorem \ref{main}(i) above). Furthermore, $c_2$ and $c_3$ belong to the tautological ring.

\end{rem}

\subsection{Blowing up the triple conic locus}\label{TCdiscussion} We would like to study the whole moduli space $\F_2$, not just $\F_2\smallsetminus \mathsf{Ell}$. Shah showed how to construct $\F_2$ as a GIT quotient \cite {S}. The discussion below is just an adaptation of his work in the language of stacks. 

The sextic given by the triple nonsingular conic, $\delta = (x_1x_3-x_2^2)^3$, plays a special role. In fact, in the quotient $[X/G]$, all surfaces in the divisor $\mathsf {Ell}$ correspond to the orbit of $\delta$. Following \cite {S}, in order to construct the moduli space $\mathcal F_2,$ we blow up the orbit of $\delta$ and remove further loci from the blowup. We make this more precise. 

Elliptic surfaces $(S, L)$ in $\mathsf{Ell}$ can also be exhibited as resolutions of certain branched double covers. In this case, the linear series $|L|=|\sigma+2f|$ has fixed part $\sigma$, and $|L-\sigma|$ induces the map $S\to Q\subset \mathbb P^2$, where $Q=\{x_1x_3-x_2^2=0\}\simeq \mathbb P^1$ is the nonsingular conic. However, the linear series $|2L|$ has no fixed components and induces a morphism $$S\to \widehat S\to \mathbb P^5\, .$$ Here $S\to \widehat S$ contracts all nonsingular rational curves on which $L$ restricts trivially. Under the Veronese embedding $$\mathsf {v}: \mathbb P^2\to \mathbb P^5,$$ the conic $Q$ corresponds to a rational normal curve $\mathsf{v}(Q)$ contained in a hyperplane section $\mathbb P^4$ of $\mathbb P^5.$ The cone over ${\mathsf v}(Q)$ is denoted $\Sigma_4^0\subset \mathbb P^5.$ This cone is a flat deformation of the Veronese surface $\mathsf v(\mathbb P^2)$. Blowing up the vertex, we obtain the Hirzebruch surface $\Sigma_4$. The surface $S$ is the resolution of the double cover $$\widehat S\to \Sigma_4^0$$ ramified over the vertex of the cone and over a curve $D$. For $S$ to be a $K3$ surface, $D$ must be a cubic section of $\Sigma_4^0\subset \mathbb P^5$ not passing through the vertex, and must have ADE singularities. 

Following \cite {La, T}, we note the following uniform description of all surfaces in $\mathcal F_2$ as resolutions of branched double covers $\widehat S$. Indeed, in all cases, $\widehat S$ arises as a complete intersection 
$$\{z^2-f_6(x_1, x_2, x_3, y)=f_2(x_1, x_2, x_3, y)=0\}\subset \mathbb P(1, 1, 1, 2, 3)\, ,$$ where $f_2, f_6$ have the indicated weighted degrees. When $f_2(0, 0, 0, 1)\neq 0,$ we can change coordinates so that $f_2(x_1, x_2, x_3, y)=y$ and we recover the sextic double planes. When $f_2(0, 0, 0, 1)=0$, and $f_2(x_1, x_2, x_3, 0)$ is non degenerate, we may assume $f_2(x_1, x_2, x_3, y)=x_1x_3-x_2^2$ and we can write $$f_6(x, y)=y^3+yg_4(x)+g_6(x)\, .$$ This corresponds to the branched double covers of $\Sigma_4^0$. Indeed, we have $\Sigma_4^0=\mathbb P(1, 1, 4)$, see \cite {D}, and the latter is cut out by $x_1x_3-x_2^2=0$ in $\mathbb P(1, 1, 1, 2).$

We return to the moduli space $\mathcal F_2$. There is a $G$-equivariant cubing map
\[
m_3:\p \Sym^2 V^*\rightarrow X\, .
\]
Let $\Delta_2\subset \p \Sym^2 V^*$ denote the divisor parametrizing singular conics. The induced map
\[
\p \Sym^2 V^*\smallsetminus \Delta_2\rightarrow X\smallsetminus (\mathsf{QP}\cup \mathsf{ML})
\]
is a $G$-equivariant closed embedding whose image is the locus $\mathsf{TC}$ parametrizing cubes of nonsingular conics. A local calculation shows that \begin{equation}\label{localcal}\mathsf{TC}\subset \mathsf{CTP} .\end{equation} All nonsingular conics are equivalent under the action of $G$. Let $W=\cc^2$ as an $\SL_2$ representation. We can identify $V=\Sym^2 W$, as the conic is a $\p^1$ embedded in $\p^2$ via the Veronese map. By Luna's \'etale slice theorem, the quotient $[X/G]$ is locally identified around $[\mathsf{TC}/G]$ with a quotient by $\PSL(W)$ of the normal slice to the orbit of the triple conic. 
The normal slice is 
\[
\Sym^8 W^*\oplus \Sym^{12} W^*\subset \Sym^6(\Sym^2 W^*)\, .
\]
This identification is more thoroughly explained in \cite[Section 5]{S}, \cite [Lemma 4.9]{La}, and in Section \ref{filcpt} below. We assign weights $2$ and $3$ to the two summands of the normal slice. \footnote{The moduli space of elliptic $K3$ surfaces requires weights $(4, 6)$ on the two summands $\Sym^8W^*$ and $\Sym^{12} W^*$, see \cite{M}. However, our calculation is carried out on the moduli of plane sextics (while ensuring that the relevant classes get identified correctly on ${\mathcal F}_2$). By doing so, we do not account for the additional $\mathbb Z/2\mathbb Z$-symmetry of the double cover. In order to match the automorphisms on the exceptional divisor, the correct weights are $(2, 3)$. The geometry here is consistent with the constructions in \cite{S}.} Next, we carry out the $(2, 3)$ weighted blowup of the locus $$\mathsf {TC}\subset X\smallsetminus (\mathsf{QP}\cup \mathsf {ML})\, ,$$ and take the quotient. 

We indicate the loci that need to be removed from the blowup to obtain $\mathcal F_2$. On the sextic ADE locus, in addition to the loci $\mathsf{ML}$ and $\mathsf {QP}$ already considered, we also need to remove the images of the double conic and double cubic locus $\mathsf {DC}$ and the consecutive triple point locus $\mathsf{CTP}$. 

Over the exceptional divisor $[\pp_{(2, 3)}(\Sym^8 W^*\oplus \Sym^{12} W^*)/\PSL(W)]$, there is an open subset corresponding to elliptic K3 surfaces. 
By \cite{M}, a pair of binary forms $(A,B)\in \Sym^8 W^*\oplus \Sym^{12} W^*$ corresponds to an elliptic K3 surface if and only if
\begin{itemize} 
\item [(i)] $4A^3+27B^2$ is not identically zero.
\item [(ii)] For each point $q\in \p^1$, the order of vanishing of $A$ at $q$ is at most $3$ or the order of vanishing of $B$ at $q$ is at most $5$.
\end{itemize}
This description was used in the calculations of \cite {CK}. We note that these are exactly the same requirements singled out in \cite [Theorem 4.3]{S}. Indeed, using the above notation, we can view $g_4$ and $g_6$ as binary forms $A, B$ of degrees $8$ and $12$ on $\mathbb P^1$, respectively. From the exact description of $\widehat S$, we find that the branch curve takes the form $D=\{y^3+yA+B=0\}$. Condition (i) corresponds to reduced branch curves $D$; this is discussed in \cite [Theorem 4.3, Case 2]{S}. Condition (ii) ensures $D$ has no consecutive triple points; this was noted in the proof of \cite [Theorem 4.3, Case (1.i)]{S} and can be seen as follows: locally near $y=u=0$, the equation $y^3+yA(u)+B(u)$ is in the ideal $(y, u^2)^3$ if and only if $A, B$ vanish with order at least $4$ and $6$ at the origin. The complement of (i) corresponds to the restriction of the strict transform of the locus $\mathsf {DC}$ to the exceptional divisor of the weighted blowup. The complement of (ii) corresponds to the restriction of the strict transform of $[\mathsf{CTP}/G]$ to the exceptional divisor. 

In conclusion, the moduli space $\mathcal F_2$ is obtained  by removing the strict transform of (the images of) $\mathsf {DC}$ and $\mathsf {CTP}$ from the weighted blowup. 

\subsection {Strategy} The following procedure will be used to complete the proof of parts (ii) and (iii) of Theorem \ref{main}. 
\begin{itemize}
    \item[(a)] We start with 
    \[
    \mathsf A_G^* (X)=\qq[H,c_2,c_3]/(p)\, ,
    \]
    where $$p=H^{28}+c_1^G(\Sym^6 V^*)H^{27}+\cdots +c_{28}^G(\Sym^6 V^*)\, .$$
    \item[(b)] We impose relations by removing $\mathsf{ML}$ (the locus of multiple lines) and $\mathsf{QP}$ (the locus of quadruple points). That is, we compute the Chow ring $\mathsf A^*_G(Y)$ where $$Y=X\smallsetminus (\mathsf{ML}\cup \mathsf{QP})\, .$$ This is carried out in Sections \ref{ML} and \ref{QP}.
    \item[(c)] We perform the weighted blowup $\widetilde{Y}$ of $Y$ at $\mathsf{TC}$ (the triple conic locus) and compute the Chow ring $\mathsf{A}^*_G(\widetilde{Y})$ . This is carried out in Section \ref{relctp}.
    \item[(d)] Finally, we impose further relations by removing from $\widetilde{Y}$ the strict transforms of $\mathsf{CTP}$ (the consecutive triple point locus) and $\mathsf{DC}$ (the union of the double conic and double cubic locus). This is carried out in Section \ref{relctp}.
\end{itemize}
From the above discussion of Shah's work, steps (a)--(d) account for all of the relations and, therefore, 
yield a presentation of
$\mathsf{A}^*(\mathcal{F}_2)$.
The calculation of the Chow Betti numbers then proves parts (ii) and (iii) of
Theorem 1.

\section{Relations from the locus of multiple lines}\label{ML}
We impose here relations obtained by removing $\mathsf{ML}$. We follow the notation from Section \ref{Shah}.
We begin by studying the map 
\[
m_1:\p \Sym^4 V^*\times \p V^*\rightarrow X=\p \Sym^6 V^*, 
\]
whose image is $\mathsf{ML}$. 
Let $h_1$ and $h_2$ denote the hyperplane classes of $\p \Sym^4 V^*$ and $\p V^*$ pulled back to the product $\p \Sym^4 V^*\times \p V^*$. As before, let $H$ be the hyperplane class of $X$. 

For notational convenience, we will
denote the map $m_1$ simply by $m$.
Moreover,
 the same symbols $m,h_1, h_2, H$ will be used to
 denote the corresponding maps and classes
 on the quotients by $G$.
 
\begin{lem}\label{nonreduced}
The image of 
\[
\mathsf A_G^{*}(\p \Sym^4 V^*\times \p V^*)\xrightarrow{m_*}\mathsf A_G^{*}(X)
\]
is the ideal generated by the three classes
\[
\sum_{j=0}^{11+i}\alpha_{i,j}H^j
\]
for $0\leq i\leq 2$, where the coefficients $\alpha_{ij}\in \mathsf A^*_G$ are recursively given by the formula
\[
\alpha_{i,k}=\gamma_*((h_1+2h_2)^{27-k}\cdot h_2^i)-\sum_{j=k+1}^{11+i} \alpha_{i,j}s^G_{j-k}(\Sym^6 V^*)\, . 
\] Here, $\gamma_*:\mathsf A_G^*(\p \Sym^4 V^*\times \p V^*)\rightarrow \mathsf A_G^*$ denotes the equivariant pushforward to a point.
\end{lem}
\begin{proof} Consider the pullback $$m^*:\mathsf A_G^*(X)\to \mathsf A_G^{*}(\p \Sym^4 V^*\times \p V^*),$$ and note that $m^*H=h_1+2h_2$. By the projective bundle formula \cite[Chapter 3]{Fulton}, the $\mathsf A_G^{*}(X)$-module $\mathsf A_G^{*}(\p \Sym^4 V^*\times \p V^*)$ is generated by the classes $1, h_2, h_2^2$. Therefore, the image of 
\[
\mathsf A_G^{*}(\p \Sym^4 V^*\times \p V^*)\xrightarrow{m_*}\mathsf A_G^{*}(X)
\]
is the ideal generated by the pushforwards $m_*(h_2^i)$ for $0\leq i \leq 2$. We write
\[
m_*(h_2^i)=\sum_{j=0}^{11+i} \alpha_{i,j}H^j\, ,
\]
where $\alpha_{i,j}\in \mathsf A^{11+i-j}_G$. We want to determine the coefficients $\alpha_{i,j}$. To pick out $\alpha_{i,k}$, multiply by $H^{27-k}$. We have
\[
H^{27-k}\cdot m_*(h_2^i)=\sum_{j=0}^{11+i}\alpha_{i,j}H^{j+27-k}\, .
\]
Consider the commutative diagram
\[
\begin{tikzcd}
{\left[\left(\p \Sym^4 V^*\times \p V^*\right)/G\right]} \arrow[rd, "\gamma"'] \arrow[r, "m"] & {[X/G]} \arrow[d, "\rho"] \\  & \, \mathrm{B}G\, .                 
\end{tikzcd}
\]
Then, 
\[
\rho_*(H^{27-k}\cdot m_*(h_2^i))=\rho_*\left(\sum_{j=0}^{11+i}\alpha_{i,j}H^{j+27-k}\right)=\sum_{j=0}^{11+i}\alpha_{i,j}s^G_{j-k}(\Sym^6 V^*)\, .
\]
On the other hand, by the projection formula,
\[
\rho_*(H^{27-k}\cdot m_*(h_2^i))=\rho_*m_*(m^*H^{27-k}\cdot h_2^i)=\gamma_*((h_1+2h_2)^{27-k}\cdot h_2^i)\, .
\]
Thus,
\[
    \sum_{j=0}^{11+i}\alpha_{i,j}s^G_{j-k}(\Sym^6 V^*)=\gamma_*((h_1+2h_2)^{27-k}\cdot h_2^i)\, .
\]
Note that $s^G_0(\Sym^6 V^*)=1$ and $s^G_{j-k}(\Sym^6 V^*)=0$ for $j-k<0$. Simplifying and rearranging, we see that
\begin{equation}\label{recursive}
\alpha_{i,k}=\gamma_*((h_1+2h_2)^{27-k}\cdot h_2^i)-\sum_{j=k+1}^{11+i} \alpha_{i,j}s^G_{j-k}(\Sym^6 V^*)\, . 
\end{equation}

\end{proof}
\begin{rem}\label{nonreducedremark}
Equation \eqref{recursive} gives a recursion for computing $\alpha_{i,k}$. The pushforwards $\gamma_{*}(h_1^a \cdot h_2^b)$ can be immediately determined via the projective bundle geometry \cite[Chapter 3]{Fulton}. Thus, we can express the classes $\alpha_{i,k}$ in terms of the generators $H,c_2,c_3$ of $\mathsf A^{*}_G(X)$. As a result, the image of $m_*$ specified by Lemma \ref{nonreduced}, and consequently the relations obtained by removing the locus $\mathsf {ML}$, can be made explicit. 

We carried out this procedure in the Macaulay2 package Schubert2 \cite{M2,S2}. The interested reader can consult \cite {COP} for the implementation. 

\noindent $\bullet$ For $i=0$,  the polynomial $\sum_{j=0}^{11}\alpha_{0,j}H^j$ is given by
\begin{align*}
  1555200c_2^4c_3&+9552816c_2c_3^3+(518400c_2^5+11162448c_2^2c_3^2)H+(5716656c_2^3c_3+56538324c_3^3)H^2\\&+(712080c_2^4+8743140c_2c_3^2)H^3+3852036c_2^2c_3H^4+(311700c_2^3+12450672c_3^2)H^5\\&-519696c_2c_3H^6+107640c_2^2H^7+243324c_3H^8-9900c_2H^9+480H^{11}\,.
\end{align*}
$\bullet$ For $i=1$, the polynomial $\sum_{j=0}^{12}\alpha_{1,j}H^j$ is given by
\begin{align*}
    1866240c_2^3c_3^2&+15431472c_3^4+ (362880c_2^4c_3+4968864c_2c_3^3)H+ (17280c_2^5+5732856c_2^2c_3^2)H^2\\&+(1278288c_2^3c_3+47364588c_3^3)H^3-(74040c_2^4+7471926c_2c_3^2)H^4\\&+ 1636848c_2^2c_3H^5-(51630c_2^3-4266918c_3^2)H^6 -598968c_2c_3H^7\\&+36810c_2^2H^8+40392c_3H^9-2850c_2H^{10}+30H^{12}\,.
\end{align*}
$\bullet$ For $i=2$, the polynomial $\sum_{j=0}^{13}\alpha_{2,j}H^j$ is given by
\begin{align*}
1259712c_2^2c_3^3&+ (1765152c_2^3c_3^2+42620256c_3^4)H-(565920c_2^4c_3-19960020c_2c_3^3)H^2\\&+(61056c_2^5+7261812c_2^2c_3^2)H^3-(426564c_2^3c_3-28062369c_3^3)H^4\\&-(15404c_2^4+8744085c_2c_3^2)H^5+1276371c_2^2c_3H^6-(63167c_2^3-1147635c_3^2)H^7\\&-218646c_2c_3H^8+12903c_2^2H^9+ 5139c_3H^{10} -389c_2H^{11}+ H^{13}\,.
\end{align*}
\end{rem}

\section{Relations from quadruple points}\label{QP}
We impose here relations obtained by removing the locus $\mathsf{QP}$ of sextics with quadruple points. A sextic $f$ has a quadruple point if locally analytically it lies in the ideal $(x,y)^4$. In this section, we will use standard material on jet bundles, see \cite [Proposition 3.2]{CanningLarson} for a review of the necessary background. 

Denote by $\pi:[\p V/G]\rightarrow \mathrm{B}G$ the universal $\p^2$-bundle and let $z$ be the hyperplane class. Then $\pi_*\O(6)=\Sym^6 V^*$ as a $G$-equivariant bundle. Its projectivization is $[X/G]$ and the projectivization of $\pi^*\pi_*\O(6)$ is $\left[\left(\p V\times X\right)/G\right]$. Consider the rank $10$ bundle of principal parts $P^3(\O(6))$ relatively over $[\p V/G]\to \mathrm{B}G$. It comes equipped with an equivariant evaluation map\[
\pi^*\Sym^6 V^*=\pi^*\pi_*\O(6)\rightarrow P^3(\O(6))\,,
\]
which on fibers takes a sextic to its expansion along a third order neighborhood. This evaluation map is surjective because $\O(6)$ is $6$-very ample and hence also $3$-very ample. The kernel is thus an equivariant vector bundle $K_{\quadr}$ of rank $18$. The bundle $K_{\quadr}$ parametrizes pairs $(f,p)$ where $f$ is a sextic with a quadruple or worse point at $p$. After projectivizing, the following diagram summarizes the above discussion:
\begin{equation}\label{quaddiagram}
\begin{tikzcd}
{[\p K_{\quadr}/G]} \arrow["\rho''",rd] \arrow[r,"j"] & {\left[\left(\p V\times X\right)/G\right]} \arrow[d,"\rho'"] \arrow[r,"\pi'"] & {[X/G]} \arrow[d,"\rho"] \\
                                       & {[\p V/G]} \arrow[r, "\pi"]               & \mathrm{B}G \,.              
\end{tikzcd}
\end{equation}
The image of $\pi'\circ j$ is the locus $[\mathsf {QP}/G]$. Moreover, since $\pi'\circ j$ is proper, the induced map on Chow groups with $\qq$-coefficients
\[
\mathsf{A}_*([\p K_{\quadr}/G])\rightarrow \mathsf{A}_*([\mathsf {QP}/G])
\]
is surjective. Indeed, one sees this by applying \cite[Lemma 3.8]{V} to the map
\[
(\p K_{\quadr}\times U)/G\rightarrow (\mathsf{QP}\times U)/G\,,
\]
where $U$ is an open subset of a representation $V$ of $G$ on which $G$ acts freely and the codimension of the complement of $U$ in $V$ is sufficiently large as in \cite[Definition-Proposition 1]{EG}. Therefore, we can compute the image of 
\[
\mathsf{A}_*([\mathsf {QP}/G])\rightarrow \mathsf{A}_*([X/G])
\]
by computing the image of
\[
\mathsf{A}_*([\p K_{\quadr}/G])\rightarrow \mathsf{A}_*([X/G])\,.
\]

\begin{lem}\label{quadrelations}
The ideal of relations obtained from removing the locus of sextics with quadruple points is generated by the classes
\[
\sum_{i=0}^{10}\rho^*\pi_*\left(z^j\cdot c^G_i(P^3(\O(6)))\right)\cdot H^{10-i}\,,
\]
where $0\leq j\leq 2$ and $z$ is the hyperplane class of $\pi$.
\end{lem}
\begin{proof}
From the explicit calculation of the Chow ring of projective bundles it follows that every class $\alpha\in \mathsf A^*_G(\p K_{\quadr})$ is a pullback of a class $\beta\in \mathsf A^*_G(\p V \times X)$. Then, by the projection formula,
\[
j_*\alpha=j_*j^*\beta= [\p K_{\quadr}]^G\cdot\beta\,.
\]
Because $\p K_{\quadr}$ is linearly embedded in $\p V\times X$, its equivariant fundamental class is given by 
\[
[\p K_{\quadr}]^G=c^G_{10}(\rho'^*P^3(\O(6))\otimes \O_{\rho'}(1))=\sum_{i=0}^{10} \rho'^*c^G_i(P^3(\O(6)))\cdot \pi'^*H^{10-i}\,.
\]
Every class $\beta\in \mathsf A^*_G(\p V \times X)$ is of the form 
\[
\beta=\beta_0+\beta_1z+\beta_2z^2\,,
\]
where $\beta_i\in \mathsf A_G^{*}(X)$ and $z$ is the hyperplane class of the projective bundle $\left[\left(\p V\times X\right)/G\right]\rightarrow [X/G]$. Hence the ideal generated by pushforwards of classes on $[\p K_{\quadr}/G]$ is just the ideal generated by the classes
\[
\pi'_*\bigg(z^j\cdot \bigg(\sum_{i=0}^{10} \rho'^*c^G_i(P^3(\O(6)))\cdot \pi'^*H^{10-i}\bigg)\bigg)=\sum_{i=0}^{10}\rho^*\pi_*(z^j\cdot c^G_i(P^3(\O(6))))\cdot H^{10-i}
\]
for $0\leq j\leq 2$, where to obtain the equality we have used \cite[Proposition 1.7]{Fulton}.
\end{proof}
\begin{rem}\label{quadremark}
The $3$ relations provided by Lemma \ref{quadrelations} can be written explicitly in terms of the generators $H,c_2,c_3$ of $\mathsf A_G^{*}(X)$. We note: 
\begin{itemize}
    \item[(i)] The equivariant Chern classes of the jet bundle $P^3(\mathcal O(6))$ are computed using the filtration by the vector bundles $\mathcal O(6)\otimes \Sym^k \Omega_{\pi}$ for $0\leq k\leq 3$ . 
    \item[(ii)] The pushforwards $\pi_* (z^j)$ are immediately found using the projective bundle geometry.
    \end{itemize}
$\bullet$ The $j=0$ case yields
\begin{align*}
-1574&64c_2c_3^2-236196c_2^2c_3H-61020c_2^3H^2+434484c_3^2
H^2+382725c_2c_3H^3+76545c_2^2H^4\\&-66339c_3H^5-13230c_2H^6+405H^8\,.
\end{align*}
$\bullet$ The $j=1$ case yields
\begin{align*}
518&40c_2^3c_3-122472c_3^3+(17280c_2^4-539460c_2c_3^2)H-446148c_2^2c_3H^2\\&-(91320c_2^3-339309c_3^2)H^3+244215c_2c_3H^4+39690c_2^2H^5-17577c_3H^6-2880c_2H^7+30H^9\,.
\end{align*}
$\bullet$ The $j=2$ case yields
\begin{align*}
     2099&52c_2^2c_3^2+(253152c_2^3c_3-338256c_3^3)H+(61056c_2^4-812592c_2c_3^2)H^2-475308c_2^2c_3H^3\\&-(76460c_2^3-178632c_3^2)H^4+104733c_2c_3H^5+13293c_2^2H^6-3267c_3H^7-390c_2H^8+H^{10}\,.
\end{align*}
\end{rem}

\section {The weighted blowup} \label{relctp}

\subsection{Overview} Lemmas \ref{nonreduced} and \ref{quadrelations} allow us to compute the Chow ring of the stack \[\left[\left(X\smallsetminus (\mathsf{ML}\cup \mathsf{QP})\right)/G\right]\,.\] The next step in the procedure is to perform the weighted blowup along the locus $\mathsf{TC}$.

\subsection{The Chow ring of weighted blowups}
First, we discuss the intersection theory of weighted blowups in general, following Arena and Obinna \cite{Arenathesis, AO}.

We will require weighted Chern and Segre classes. Let $E\to B$ be a vector bundle over a base $B$ equipped
with a fiberwise action of
$\mathbb{G}_m$. We define $$P_{E}(t)=c^{\mathbb G_m}_{\text{top}}(E)\, ,$$  a polynomial of degree $\text{rk }E$ in the equivariant parameter $t$. In case the $\mathbb{G}_m$-action is of weight $1$, $P_E(t)$ recovers the total Chern polynomial of $E$. For the case of $(2, 3)$-weighted blowups needed below, $E=E_2\oplus E_{3}$ is split into subbundles of weights $2$, $3$ respectively.  Then, we have $$P_E(t)=\prod_{i=1}^{\text{rk }E_2}(2t+\alpha_i) \cdot \prod_{i=1}^{\text{rk } E_{3}}(3 t+\beta_i)\, ,$$ where $\alpha_i$ and  $\beta_i$ are the Chern roots of $E_2$ and $E_{3}$. The weighted Segre class of $E$ is given by $$s^{\mathsf{wt}}(E)=\frac{1}{P_E(1)}\, ,$$ see \cite[Remark 8.1.2]{Arenathesis}.  

For the blowup geometry, let $i:Z\hookrightarrow Y$ be a closed embedding of codimension $d$ with $Z$ and $Y$ both nonsingular. Let $N$ be the normal bundle weighted by a $\gg_m$-action. Let $P_N(t)$ be the $\gg_m$-equivariant top Chern class $c_d^{\gg_m}(N)$.

Let $f:\tilde{Y}\rightarrow Y$ denote the corresponding weighted blowup along $Z$, and $j:\tilde{Z}\hookrightarrow \tilde{Y}$ the exceptional divisor. The exceptional divisor $g:\tilde{Z}\to Z$ is a weighted projective bundle.
We have $$\mathsf A^*(\tilde{Z})=\mathsf A^*(Z)[t]/(P_N(t))$$ where $t$ is set to the hyperplane class of $g$. 

By \cite[Theorem 5.5]{AO}, there is a commutative diagram 
\[
\begin{tikzcd}
\mathsf A^*(\tilde{Z}) \arrow[r, "j_*"]          & \mathsf A^*(\tilde{Y})          \\
\mathsf A^*(Z) \arrow[r, "i_*"] \arrow[u, "f^!"] & \mathsf A^*(Y) \arrow[u, "f^*"]
\end{tikzcd}
\]
where $f^{!}(\alpha)=g^*\alpha\cdot \delta$. Here, $$\delta=\frac{P_N(t)-P_N(0)}{t}\,.$$ Furthermore, we have the following structure theorem for the Chow groups of weighted blowups, which we present in the simplified setting where the pullback $\mathsf A^*(Y)\rightarrow \mathsf A^*(Z)$ is surjective \cite [Corollary 6.5]{AO}. When all weights are $1$, Theorem \ref{blowupsequence} below recovers \cite[Appendix, Theorem 1]{Keel}.
\begin{thm}[Arena-Obinna]\label{blowupsequence}
    In the above setting, if $i^*:\mathsf A^*(Y)\rightarrow \mathsf A^*(Z)$ is surjective, then
    \[
\mathsf{A}^*(\tilde{Y})\cong \mathsf{A}^*(Y)[t]/(t\cdot \ker(i^*),Q(t))\,,
    \]
    where
    $Q(t)=P_N(t)-P_N(0)+[Z]$. Furthermore, $[\tilde{Z}]=-t$.
\end{thm}

Let $\zeta: S\hookrightarrow Y$ be a nonsingular closed substack of dimension $k$ such that $S\cap Z$ is also nonsingular. Let $\tilde{S}\subset \tilde{Y}$ be the strict transform inside the weighted blowup. We will use a formula for $[\tilde{S}]\in \mathsf{A}_k(\tilde{Y})$ in terms of the total transform $f^*[S]$ and a class pushed forward from the exceptional divisor (generalizing \cite[Theorem 6.7]{Fulton} in the unweighted case). The exceptional term involves the weighted top Chern class $P_N$ of the weighted normal bundle and the weighted Segre class of the normal bundle $N_{S/S\cap Z}$. The following result is \cite [Theorem 8.2.1]{Arenathesis}.

\begin{thm}\label{strict}
    With notation as above, we have
    \[
    [\tilde{S}]=f^*[S]-j_*\{P_N(1)(1+t+t^2+\dots)\cdot g^*\zeta_*s^{\mathsf{wt}}(N_{S/S\cap Z})\}_k\,,
    \]
    where $-t$ is the class of the exceptional divisor. 
\end{thm}
\noindent Here, the subscript $k$ denotes taking the dimension $k$ piece of the term in brackets.

\subsection{The weighted blowup of the locus of triple conics} We apply these results to our situation. Set $Y=X\smallsetminus (\mathsf{ML}\cup \mathsf{QP})$. We have the diagram
\[
\begin{tikzcd}
\mathsf A_G^{*}(\tilde{\mathsf{TC}}) \arrow[r, "j_*"]          & \mathsf A_G^{*}(\tilde{Y})          \\
\mathsf A_G^{*}(\mathsf{TC}) \arrow[r, "i_*"] \arrow[u, "f^!"] & \mathsf A_G^{*}(Y) \arrow[u, "f^*"]
\end{tikzcd}
\]

\begin{cor}\label{Ytilde}
With notation as above, we have
        \[
\mathsf{A}_G^*(\tilde{Y})\cong \mathsf{A}_G^*(Y)[t]/(t\cdot H,t\cdot c_3,Q(t))\,,
    \]
    where
    $Q(t)=P_N(t)-P_N(0)+[\mathsf{TC}]$. Furthermore, $[\tilde{\mathsf{TC}}]=-t$.
\end{cor}
\begin{proof}
We note that $$i^*:\mathsf A^*_G(Y)\to \mathsf A^*_G(\mathsf{TC}), \quad i^*H=0, \quad i^*c_3=0\,.$$ Indeed, for the vanishing of $i^*H$, we recall that $\mathsf{TC}=\mathbb P^5\smallsetminus \Delta_2,$ where $\Delta_2$ is the divisor of singular conics. 
The vanishing of $i^*c_3$ is a consequence of the fact that the stabilizer of the orbit $\mathsf{TC}$ is $\PSL_2.$ In fact, $$\mathsf {A}_G^{*}(\mathsf{TC})=\mathsf{A}^*_G(G/\mathsf{PSL}_2)=\mathsf{A}^*_{\mathsf{PSL}_2}(\text{pt})=\mathbb Q[c_2].$$ Furthermore, using the map $\mathsf{PSL}_2\to G$, $g\to \text{Sym}^2g$, we see that $i^*c_2=4c_2.$ These remarks show 
 $$\ker i^*=\langle H, c_3\rangle\,$$ 
and that $i^*$ is surjective. Applying Theorem \ref{blowupsequence} yields the result.
\end{proof}

\begin{rem}\label{twomore} The Chow ring $\mathsf{A}^*_G(Y)$ is 
has a presentation with 
generators $H,c_2,c_3$ 
and 6 relations (obtained from Remarks \ref{nonreducedremark} and \ref{quadremark}). Applying Corollary \ref{Ytilde} then yields a complete presentation of $\mathsf{A}^*_G(\tilde{Y})$ by imposing $3$ additional relations. 
After restriction to $\mathcal{F}_2$, 
$$t=-[\mathsf{Ell}]\,.$$
The polynomial $P_N$, required for the third relation of Corollary \ref{Ytilde}, will be recorded in equation \eqref{no} below. However, for degree reasons, when restricting to $\mathcal F_2$, only the first two relations $$\left[\mathsf {Ell}\right]\cdot H=0\, , \quad \left[\mathsf{Ell}\right]\cdot c_3=0$$ will be meaningful. \end{rem}

\subsection{A resolution of $\mathsf{CTP}$}

By the discussion in Section \ref{Shah}, we have morphisms $$\mathcal F_2\stackrel{\zeta}{\longrightarrow}[\tilde Y/G]\stackrel{f}{\longrightarrow} [Y/G]\,,$$ where the complement of the image of $\zeta$ is the union of $\widetilde{\mathsf {DC}}$ and $\widetilde {\mathsf{CTP}}.$ As noted in the previous paragraph, we have a complete description of $\mathsf{A}^*_G(\widetilde {Y})$. 
To understand $\mathsf{A}^*(\F_2)$, we first compute relations by removing $\tilde{\mathsf{CTP}}$.

An additional complication arises here. Theorem \ref{strict} does not directly apply to the blowup $\tilde{Y}$ because $\mathsf{CTP}$ is singular. In the following subsection, we will use Theorem \ref{strict} for a resolution of $\mathsf{CTP}$, which we now construct. 

Recall that a sextic $f$ has a consecutive triple point if analytically locally it lies in the ideal $(x,y^2)^3$. We begin by constructing the relevant bundle of principal parts. Note that the local equation is in the ideal $(x,y^2)^3$ if and only if the coefficients in the Taylor expansion of the monomials in the set  $$S=\{1,x,y,x^2,xy,y^2,x^2y,xy^2,y^3,xy^3,y^4,y^5\}$$ all vanish. To record the data of these monomial coefficients, we use the machinery of refined principal parts bundles as in \cite[Section 3.2]{CanningLarson}. The universal $\p^2$ bundle is denoted by $\pi:[\p V/G]\rightarrow \mathrm{B}G$. Set $T$ to be the tangent bundle of $\p V$. The set $S$ is admissible in the sense of \cite[Definition 3.7]{CanningLarson}. There is a rank $12$ bundle on the domain of $a:[\p T/G]\rightarrow [\p V/G]$ denoted $P^S(\O(6))$ and an evaluation morphism
\begin{equation}\label{refined}
a^{*}\pi^*\Sym^6 V^*=a^*\pi^*\pi_*\O(6)\rightarrow  a^* P^5(\O(6))\rightarrow P^S(\O(6))\,,
\end{equation}
where the first map is the usual fifth order principal parts evaluation pulled back to $\p T$, and the second map truncates the Taylor series along the monomials in $S$. The composite is surjective, as the first map is surjective because $\O(6)$ is $5$-very ample, and the second map is surjective by definition. 
Let $K_{\mathrm{ctp}}$ denote the kernel of the morphism \eqref{refined}:
\begin{equation}\label{k}0\to K_{\text{ctp}}\to a^{*}\pi^*\Sym^6 V^*\rightarrow P^S(\O(6))\to 0\,.\end{equation}
It is a $G$-equivariant vector bundle of rank $16$ on $\pp T\times X$, parametrizing pairs $(f,p,v)$ where $f$ is a sextic with a consecutive triple point at $p$ in the tangent direction $v\in T_{\,p\,} \mathbb PV.$ 

Denote by $\pp K^{\circ}$ the restriction of $\pp K_{\ctp}$ to $\pp T\times Y$, where $Y=X\smallsetminus (\mathsf{QP}\cup \mathsf{ML})$. The following diagram summarizes the situation:
\begin{equation}\label{summaryctp}
\begin{tikzcd}
{[\p K^\circ/G]} \arrow[r, "\iota"] \arrow[rd, "\rho_3"']     & {\left[\left(\p T\times Y\right)/G\right]} \arrow[r, "a'"] \arrow[d, "\rho_2"]       & {\left[\left(\p V \times Y\right)/G\right]} \arrow[d, "\rho_1"] \arrow[r, "\pi'"]       & {[Y/G]} \arrow[d, "\rho"]      \\
                                                              & {[\p T/G]} \arrow[r, "a"]                                          & {[\p V/G]} \arrow[r, "\pi"]                                       & \mathrm{B}G\,.                           
\end{tikzcd}
\end{equation}
The image of $\pi'\circ a'\circ\iota$ is the locus $\mathsf{CTP}\cap Y$, which by abuse of notation we will simply refer to as $\mathsf{CTP}$. 

We want to impose relations from removing $\tilde{\mathsf{CTP}}$, not just $\mathsf{CTP}$. We perform the weighted blowup of $Y$ along $\mathsf{TC}$ to obtain $\tilde{Y}$. Because $\pi$ is flat, the weighted blowup of $\pp V\times Y$ along $\pi'^{-1}(\mathsf{TC})$ is simply $\pp V\times \tilde{Y}$. Similarly, the weighted blowup of $\pp T \times Y$ along $a'^{-1}\pi'^{-1}(\mathsf{TC})$ is $\pp T\times \tilde{Y}$. Let $\tilde{\pp K^{\circ}}$ be the strict transform of $\pp K^{\circ}$. We have the following commutative diagram, in which each cell in the right most $3\times 3$ square in the blowup diagram is Cartesian:
\begin{equation}\label{bigdiagram}
\begin{tikzcd}
{[\tilde{\pp K^{\circ}}/G]} \arrow[d, "h"] \arrow[r, "\iota_{\mathsf{str}}"] & {[\pp T \times \tilde{Y}/G]} \arrow[d, "f_2"] \arrow[r, "a''"] & {[\pp V\times \tilde{Y}/G]} \arrow[d, "f_1"] \arrow[r, "\pi''"] & {[\tilde{Y}/G]} \arrow[d, "f"] \\
{[\pp K^{\circ}/G]} \arrow[r, "\iota"] \arrow[rd, "\rho_3"']                 & {[\pp T \times Y /G]} \arrow[r, "a'"] \arrow[d, "\rho_2"]      & {[\pp V \times Y/G]} \arrow[r, "\pi'"] \arrow[d, "\rho_1"]     & {[Y/G]} \arrow[d, "\rho"]      \\
                                                                      & {[\pp T/G]} \arrow[r, "a"]                                     & {[\pp V/G]} \arrow[r, "\pi"]                                   & \mathrm{B}G \,.                  
\end{tikzcd}
\end{equation}
The next lemma shows we can compute relations from removing $\tilde{\mathsf{CTP}}$ by finding the image of the proper pushforward $\mathsf{A}^G_*(\tilde{\pp K^{\circ}})\rightarrow \mathsf{A}^G_*(\tilde Y)$ along the morphism $\pi''\circ a''\circ\iota_{\mathrm{str}}$.

\begin{lem}\label{lp44}
    The map $\pi''\circ a''\circ\iota_{\mathrm{str}}:[\tilde{\pp K^{\circ}}/G]\rightarrow [\tilde{Y}/G]$ is proper and its image is $\tilde{\mathsf{CTP}}$. Therefore, the induced proper pushforward map
    \[
    \mathsf{A}^G_*(\tilde{\pp K^{\circ}})\rightarrow \mathsf{A}^G_*(\tilde{\mathsf{CTP}})
    \]
    is surjective.
\end{lem}
\begin{proof}
    By the universal property of blowing up, $\pi''\circ a''\circ\iota_{\mathrm{str}}$ factors through $\tilde{\mathsf{CTP}}$. It is proper because $\iota_{\mathrm{str}}$ is a closed embedding and $a''$ and $\pi''$ are projective bundles. Moreover, it is dominant because $[\pp K^{\circ}/G]\rightarrow [Y/G]$ surjects onto $\mathsf{CTP}$. The first statement follows. For the surjectivity of the pushforward $(\pi''\circ a''\circ\iota_{\mathrm{str}})_*$, we apply the same argument given before Lemma \ref{quadrelations}.
\end{proof}
\subsection{Chow classes on $\mathsf{CTP}$ and its resolution} In order to apply Lemma \ref{lp44}, we study the Chow rings of the the spaces in diagram \eqref{bigdiagram}. We let $$z=c_1(\O_{\pi}(1)), \quad \tau=c_1(\O_{a}(1)).$$ Because $a'$ and $a''$ (respectively, $\pi'$ and $\pi''$) are pullbacks of $a$ (respectively, $\pi$), we will omit pullbacks and use $\tau$ (respectively, $z$) also for the hyperplane classes of $a'$ and $a''$ (respectively, $\pi'$ and $\pi''$). By the projective bundle formula and Corollary \ref{Ytilde}, the Chow ring of every space in the rightmost $3\times 3$ square in diagram \eqref{bigdiagram} is understood. 

Moreover, because $\pp K^{\circ}$ is open inside a projective bundle over $\pp T$, the Chow ring $\mathsf{A}^*_G(\pp K^{\circ})$ is generated over $\mathsf{A}^*_G$ by the pullbacks of $z,\tau$ and the hyperplane class of $\rho_3$. Because $\pp K_{\ctp}$ is linearly embedded in $\pp (a^*\pi^*\Sym^6 V^*)=\pp T \times X$, the hyperplane class of $\rho_3$ is the restriction of the hyperplane class of $\rho_2$, which in turn is the pullback of $H$ along $a'\circ \pi'$. In particular, every class in $\mathsf{A}^*_G(\pp K^{\circ})$ is pulled back from $\mathsf{A}^*_G(\pp T\times Y)$.

To determine generators for $\mathsf{A}^*_G(\tilde{\pp K^{\circ}})$, we note that $\tilde{\pp K^{\circ}}$ is the weighted blowup of $\pp K^{\circ}$ along the incidence correspondence $J=\mathsf{TC}\times_{\mathsf{CTP}} \pp K^{\circ}$. Thus, $\mathsf{A}^*_G(\tilde{\pp K^{\circ}})$ can be computed via Theorem \ref{blowupsequence}. The answer is recorded in Corollary \ref{modulegenerators} below. 

We first note that $J$ parametrizes tuples $(g,p,v)$ such that $g$ is a nonsingular conic with a point $p$ and a tangent vector $v$ to $g$ at $p$. We can describe $J$ using a bundle of principal parts, similarly to the way we formed $K_{\ctp}$. Let $S'=\{1,y\}$ and form the bundle of principal parts $P^{S'}(\O(2))$ in the sense of \cite[Definition 3.7]{CanningLarson}. This is a rank $2$ bundle on $\pp T$ equipped with a $G$-equivariant evaluation morphism
\[
a^*\pi^*\Sym^2 V^*\rightarrow a^*P^1(\O(2))\rightarrow P^{S'}(\O(2))\,,
\]
which is surjective because $\O(2)$ is $2$-very ample. We denote the kernel by $U$. The space $\pp U$ parametrizes triples $(g,p,v)$ such that $g$ is a conic with a point $p$ and tangent vector $v$ to $g$ at $p$. Thus, $J\subset \pp U$ is the open subspace corresponding to when the conic is nonsingular. Note that $J$ itself is nonsingular because it is open inside of a projective bundle $\pp U$ over $\pp T$.

\begin{lem}\label{pullbacksurjection}
    The Chow ring $\mathsf{A}_G^*(J)$ is generated by the restrictions of $c_2$, $z$ and $\tau$ to $J$. In particular, the pullback map
    \[
    \mathsf{A}_G(\pp K^{\circ})\rightarrow \mathsf{A}_G^*(J)
    \]
    is surjective.
\end{lem}
\begin{proof}
    By definition, $\pp U$ is a projective bundle over $\pp T$, and so $\mathsf A^*_G(\pp U)$ is generated by (the pullbacks of) $c_2,c_3,z,\tau$ and $\zeta=c_1(\O_{\pp U}(1))$. 
    Moreover, $\pp U\rightarrow \pp(a^*\pi^*\Sym^2 V^*)=\pp \Sym^2 V^*\times \pp T $ is a linear embedding, under which $J$ maps to the locus of nonsingular conics $(\pp \Sym^2 V^*\smallsetminus \Delta_2)\times \pp T$. The classes $c_2,c_3,z,\tau,$ and $\zeta$ are all pulled back from $\pp \Sym^2 V^*\times \pp T$, and on $(\pp \Sym^2 V^*\smallsetminus \Delta_2)\times \pp T$, the hyperplane class and $c_3$ vanish. Therefore, $\mathsf A^*_G(J)$ is generated by the restrictions of $c_2, z, \tau$. 
\end{proof}
\begin{cor}\label{modulegenerators}
    The Chow ring $\mathsf{A}^*_G(\tilde{\pp K^{\circ}})$ is generated as an $\mathsf{A}^*_G(\tilde{Y})$ module by the restrictions of $\tau$ and $z$. Hence, the image of the pushforward map
    \[
    \mathsf{A}_G^*(\tilde{\pp K^{\circ}})\rightarrow \mathsf{A}_G^*(\tilde{Y})
    \]
    is generated by the classes
    \[
    (\pi''\circ a'')_*([\tilde{\pp K^{\circ}}]^G\cdot \tau^j z^i)
    \]
    for $0\leq j \leq 1$ and $0\leq i \leq 2$.
\end{cor}
\begin{proof}
    By Theorem \ref{blowupsequence}, $\mathsf{A}^*_G(\tilde{\pp K^{\circ}})$ is generated as an algebra by pullbacks from $\mathsf{A}_G^*(\pp K^{\circ})$ and the fundamental class of the exceptional divisor. Because the exceptional divisor of $\tilde{\pp K^{\circ}}$ is the restriction of the exceptional divisor of $\pp T\times \tilde{Y}$, every class on $\tilde{\pp K^{\circ}}$ is pulled back from $\pp T\times \tilde{Y}$. The Chow ring $\mathsf{A}^*_G(\pp T\times \tilde{Y})$ is generated as an $\mathsf{A}^*_G(\tilde{Y})$ module by $\tau$ and $z$. The second statement follows from the first by the projection formula.
\end{proof}
By Lemma \ref{lp44}, to compute $\mathsf{A}^*_G(\tilde{Y}\smallsetminus \tilde{\mathsf{CTP}})$, we need to make explicit the classes in Corollary \ref{modulegenerators} in terms of the generators of $\mathsf{A}^*_G(\tilde{Y})$.


\subsection {Relations from consecutive triple points}\label{filcpt} To compute the pushforwards in Corollary \ref{modulegenerators} in terms of the generators for $\mathsf{A}^*_G(\tilde{Y})$, we need a more explicit formula for the class $[\tilde{\pp K^{\circ}}]^G$. We use Theorem \ref{strict}, the weighted blowup formula, applied to the weighted blowup $f_2:\pp T \times \tilde{Y}\rightarrow \pp T \times Y$ along $\pp T\times \mathsf{TC}$. We denote by $i_2:\pp T\times \mathsf{TC}\rightarrow \pp T \times Y$ the inclusion. We have the weighted blowup diagram
\[
\begin{tikzcd}
 \pp T \times \tilde{\mathsf{TC}} \arrow[r, "j_2"] \arrow[d,"g_2"]        & \pp T \times \tilde{Y} \arrow[d,"f_2"]          \\
 \pp T \times \mathsf{TC} \arrow[r, "i_2"] & \mathsf \pp T \times Y \,.
\end{tikzcd}
\]
Let $\xi: J\rightarrow \pp T \times \mathsf{TC}$ be the inclusion. By Theorem \ref{strict}, 
\begin{equation}\label{strictformula}
[\tilde{\pp K^{\circ}}]^G=f_2^{*}[\pp K^{\circ}]^G-j_{2*}\{P_N(1)(1+t+t^2+\dots)\cdot g_2^*\xi_*s^{\mathsf{wt}}(N_{J/\pp K^{\circ}})\}_{18}\,,
\end{equation} where $-t$ is the class of the exceptional divisor.

We analyze the two terms on the right hand side of equation \eqref{strictformula} separately. The fundamental class $[\pp K^{\circ}]$ can be determined analogously to the computation of $[\pp K_{\quadr}]$ in Section \ref{QP}. Because $\p K^{\circ}$ is linearly embedded in $\p T\times Y$, its equivariant fundamental class is given by 
\[
[\p K^{\circ}]^G=c^G_{12}(\rho_2^*P^S(\O(6))\otimes \O_{\rho_2}(1))=\sum_{k=0}^{12} \rho_2^*\,c^G_k(P^S(\O(6)))\cdot (\pi'\circ a')^*H^{12-k}\,.
\]

\begin{rem}\label{r15}
    Using the above calculation of the equivariant fundamental class, the $6$ classes
    \[
    (\pi''\circ a'')_*(f_2^{*}[\pp K^{\circ}]^G\cdot \tau^j z^i) = f^* (\pi' \circ a')_*([\pp K^{\circ}]^G\cdot \tau^j z^i)
    \]
    for $0\leq j \leq 1$ and $0\leq i \leq 2$ can be written explicitly in terms of the $f$ pullbacks of the generators $H,c_2,c_3$ of $\mathsf A^*_G(Y)$. The equivariant Chern classes of $P^S(\O(6))$ are computed using that $P^S(\O(6))$ is filtered by tensor products of symmetric powers of the tautological bundles on $\p T$ and the bundle $\O(6)$. More explicitly, we have the tautological sequence
    \[
    0\rightarrow \O_{\p T}(-1)\rightarrow a^*T\rightarrow Q\rightarrow 0\,.
    \]
    We set $\Omega_x=\O_{\p T}(1)$ and $\Omega_y=Q^*$. Then for each monomial $x^iy^j$ in $S$, the filtration will have successive quotients $\Sym^i\Omega_x\otimes \Sym^j\Omega_y\otimes \O(6)$. For more details on this filtration, see \cite[Section 3.2]{CanningLarson}. In any case, $$c_1(\Omega_x)=\tau\, , \quad c_1(\Omega_y)=-3z-\tau\, ,$$ and the Chern classes of $P^S(\mathcal O(6))$ follow from here. 
    
    We have implemented the calculation in the Macaulay2 package Schubert2 \cite{M2,S2}: 
    
    \noindent $\bullet$ For $i=0$ and $j=0$, we obtain
    \begin{align*}
            -36288c_2^3c_3&-244944c_3^3-(254592c_2^4+2610792c_2c_3^2)H-1154736c_2^2c_3H^2\\&+(848280c_2^3+4719870c_3^2)H^3+883548c_2c_3H^4-588546c_2^2H^5+61236c_3H^6\\&+118620c_2H^7-4362H^9\,.
    \end{align*}
    $\bullet$ For $i=1$ and $j=0$, we obtain
\begin{align*}   
34560c_2^5&+233280c_2^2c_3^2+(83376c_2^3c_3-2350296c_3^3)H-(623856c_2^4+6328044c_2c_3^2)H^2\\&-1198476c_2^2c_3H^3+(810240c_2^3+2435751c_3^2)H^4-180306c_2c_3H^5-316413c_2^2H^6\\&+84186c_3H^7+28950c_2H^8-381H^{10}\,.
 \end{align*}
$\bullet$    For $i=2$ and $j=0$, we obtain
\begin{align*}
70848c_2^4c_3&+478224c_2c_3^3+(254592c_2^5+2960712c_2^2c_3^2)H+(525240c_2^3c_3-4865832c_3^3)H^2\\&-(847664c_2^4+6647562c_2c_3^2)H^3-93798c_2^2c_3H^4+(589876c_2^3+567756c_3^2)H^5\\&-395280c_2c_3H^6-117822c_2^2H^7+26550c_3H^8+4432c_2H^9-14H^{11}\,.
\end{align*}
    $\bullet$ For $i=0$ and $j=1$, we obtain
\begin{align*}
-69120c_2^5&-466560c_2^2c_3^2-(616896c_2^3c_3-1032264c_3^3)H+ (1181376c_2^4+10152540c_2c_3^2)H^2\\&+2659392c_2^2c_3H^3-(1697460c_2^3+5233653c_3^2)H^4-441774c_2c_3H^5+623943c_2^2H^6\\&-30942c_3H^7-56070c_2H^8+831H^{10}\,.
    \end{align*}
   $\bullet$ For $i=1$ and $j=1$, we obtain
    \begin{align*}
        31104c_2^4c_3&+209952c_2c_3^3-(463104c_2^5+4805568c_2^2c_3^2)H-(1385856c_2^3c_3-5604552c_3^3)H^2\\&+(1718688c_2^4+12248496c_2c_3^2)H^3+845640c_2^2c_3H^4-(1212072c_2^3+1992276c_3^2)H^5\\&+343116c_2c_3H^6+225864c_2^2H^7-30564c_3H^8-9024c_2H^9+48H^{11}\,.
    \end{align*}
    $\bullet$ For $i=2$ and $j=1$, we obtain
    \begin{align*}
        69120c_2^6&+575424c_2^3c_3^2+734832c_3^4+(149472c_2^4c_3-4974696c_2c_3^3)H\\&-(1180896c_2^5+12642480c_2^2c_3^2)H^2-(953352c_2^3
        c_3-7605414c_3^3)H^3\\&+(1698256c_2^4+7821144c_2c_3^2)
        H^4-773712c_2^2c_3H^5-(623858c_2^3+346977c_3^2)H^6\\&+
        263682c_2c_3H^7+55773c_2^2H^8-7110c_3H^9-896c_2H^{10}
        +H^{12}\,.
    \end{align*}
\end{rem}

Next, we compute the contributions from the the second term
\[
j_{2*}\{P_N(1)(1+t+t^2+\dots)\cdot g_2^*\xi_*s^{\mathsf{wt}}(N_{J/\pp K^{\circ}})\}_{18}\,
\]
in equation \eqref{strictformula}. First, recall that the polynomial $P_N$ is the top weighted Chern class of the weighted normal bundle $N_{\pp T\times \mathsf{TC}/\pp T \times Y}=\text{pr}^*N_{\mathsf{TC}/Y}$. We recall from Section \ref{TCdiscussion} that 
\[N_{\mathsf{TC}/Y}=\Sym^8 W^*\oplus \Sym^{12} W^*\]
with weights $2$ and $3$, where $W$ is a rank $2$ bundle with trivial first Chern class such that $\Sym^2 W = V$. In particular, $c^G_2(W)=\frac{1}{4}c_2$. Therefore,
\[
P_N(t)=(\pi'\circ a')^*\left(\left(\sum_{i=0}^9 c^G_i(\Sym^8 W^*)(2t)^{9-i}\right)\left(\sum_{i=0}^{13} c^G_i(\Sym^{12} W^*)(3t)^{13-i}\right)\right)\,.
\]
In fact, \begin{align}\label{no}\,&\,\,\,\,P_N(t)=816293376\, t^{22} + 14375833344\, c_2 t^{20} + 106093574496 \,c_2^2 t^{18} \\ 
 &+428059655424\, c_2^3 t^{16} + 1034306024256\, c_2^4 t^{14} + 
 1543171803264\, c_2^5 t^{12} + 1416080524896\, c_2^6 t^{10} \nonumber \\
&+ 772698973824\, c_2^7 t^8+ 233467992576\, c_2^8 t^6 + 34249734144 c_2^9 t^4 + 
 1791590400\, c_2^{10} t^2\,.\nonumber\end{align}
Next, we calculate the weighted Segre class $s^{\mathsf {wt}}(N_{J/\pp K^{\circ}})$. Because $J$ is open in $\pp U$, $N_{J/\pp K^{\circ}}$ is just the restriction of $N_{\pp U/\pp K^{\circ}}$ to $J$. We first discuss the weight decomposition of this bundle. 

The bundle $K_{\ctp}$ was constructed as a subbundle
\[
K_{\ctp}\hookrightarrow a^*\pi^* \Sym^6 V^*
\]
corresponding to monomials in the set $S=\{1,x,y,x^2,xy,y^2,x^2y,xy^2,y^3,xy^3,y^4,y^5\}$, which are the monomials that must vanish for the sextic to lie in $(x,y^2)^3$. The filtration determining the weights is the restriction of the filtration determining the weights on $a^*\pi^*\Sym^6 V^*$. This filtration was constructed in \cite[Section 5]{S}. We fix the equation $Q$ of a nonsingular conic and obtain a filtration 
\[
\Sym^6 V^* \supset Q \cdot \Sym^4 V^*\supset Q^2\cdot \Sym^2 V^* \supset Q^3\cdot \mathbb{C} \,.
\] This corresponds to the decomposition of each sextic into pieces $$f=f_6+f_4 \cdot Q + f_2 \cdot Q^2+f_0\cdot Q^3.$$ The relevant successive subquotients are $$\Sym^6 V^*/\Sym^4 V^*\cong \Sym^{12} W^*, \quad \Sym^4 V^*/\Sym^2 V^*\cong \Sym^8 W^*,$$ which correspond to picking out the terms $f_6|_{Q}$ and $f_4|_{Q}$. The restriction of this filtration to $K_{\ctp}$ is denoted
\[
K_{\ctp}\supset Q\cdot K_{4}\supset Q^2\cdot K_{2} \supset Q^3 \cdot K_{0}\,.
\]

We can describe the pieces of the filtration more explicitly in terms of principal parts. In local coordinates, we write the equation for the nonsingular conic as $Q=x-y^2$. Then $K_{0} $ is spanned by $Q^3$. The bundle $K_{2}$ is given by the degree $2$ polynomials $h$ such that $$Q^2 h\in (x,y^2)^3.$$ Similarly, $K_{4}$ is given degree $4$ polynomials $h$ such that $$Qh\in (x,y^2)^3.$$ Therefore, sections of $K_{2}$ are degree $2$ polynomials with monomials in the set $\{x,x^2,xy,y^2\}$, and sections of $K_{4}$ are degree $4$ polynomials with monomials in $\{x^2,x^3,x^2y,xy^2,x^4,x^3y,xy^3,x^2y^2,y^4\}$. In other words, setting $S'=\{1,y\}$, we have an exact sequence
\begin{equation}\label{k2}
0\rightarrow K_2\rightarrow a^*\pi^*\Sym^2 V^*\rightarrow P^{S'}(\O(2))\rightarrow 0\,,
\end{equation}
where $P^{S'}(\O(2))$ is the bundle of principle parts along $S'$ in the sense of \cite[Definition 3.7]{CanningLarson}. Note that this was exactly the way in which $U$ was defined, so $K_2\cong U$. Similarly, we have
\begin{equation}\label{k4}
0\rightarrow K_4\rightarrow a^*\pi^*\Sym^4 V^*\rightarrow P^{S''}(\O(4))\rightarrow 0\,,
\end{equation}
where $S''=\{1,x,y,xy,y^2,y^3\}$. 
Therefore, the normal bundle $N_{J/\pp K^{\circ}}$ has the same weighted Chern class as
\[
K_4/K_2 \oplus K_{\ctp}/K_4
\]
with weight $2$ on the first factor and weight $3$ on the second. Here we have omitted the restriction to $J$ from the notation. Using the defining exact sequences \eqref{k2}, \eqref{k4}, \eqref{k} for $K_2, K_4$, and $K_{\ctp}$ allows us to compute the weighted Segre class $s^{\mathsf{wt}}(N_{J/\pp K^{\circ}})$ using standard Chern class manipulations. The resulting expression is explicit, but a bit long, so we record only a few terms $$s^{\mathsf{wt}}(N_{J/\pp K^{\circ}})=\frac{1}{2^{5}\cdot 3^{7}}\left(1+(2z-5\tau)+\frac{1}{36} (-2536 c_2 + 285 \tau^2 - 869 \tau z - 603 z^2)+\ldots\right).$$ Higher order terms are calculated using Macaulay2 \cite{COP}.

Next, we need to compute the pushforward $\xi_*s^{\mathsf{wt}}(N_{J/\pp K^{\circ}})$. Because every class on $J$ is pulled back from $\pp T \times \mathsf{TC}$ as was shown in the proof of Lemma \ref{pullbacksurjection}, the pushforward $\xi_*s^{\mathsf {wt}}(N_{J/\pp K^{\circ}})$ is given by multiplication with the fundamental class of $J$. The fundamental class of $J$ is given by the the top Chern class
\[
c^G_{2}(P^{S'}(\O(2))\otimes \O_{\pp U}(1))\,
\]
restricted to $J$. In fact, $\mathcal O_{\pp U}(1)$ is trivial since the hyperplane class is trivial on $\mathsf{TC}$, and thus $$[J]=c_2(P^{S'}(\O(2)))=2z\cdot (2z+c_1(\Omega_y))=2z\cdot (2z+(-3z-\tau))\,.$$

\begin{rem}\label{r16}
We have now described each term in the expression
\[
j_{2*}\{P_N(1)(1+t+t^2+\dots)\cdot g_2^*\xi_*s^{\mathsf{wt}}(N_{J/\pp K^{\circ}})\}_{18}\,.
\]
Using Macaulay2 \cite{COP}, we have computed the contributions coming from the exceptional terms:

\noindent $\bullet$ For $i=0$ and $j=0$, we obtain 
$$-398016\,c_2^4\,\left[\mathsf{Ell}\right] + 5472720\,c_2^3\,\left[\mathsf{Ell}\right]^3 - 11144448\,c_2^2\,\left[\mathsf{Ell}\right]^5 + 
 4477680\,c_2\,\left[\mathsf{Ell}\right]^7 - 279936\,\left[\mathsf{Ell}\right]^9\,.$$

\noindent $\bullet$ For $i=1$ and $j=0$, we obtain
$$1961568\,c_2^4\,\left[\mathsf{Ell}\right]^2-9611640\,c_2^3\,\left[\mathsf{Ell}\right]^4+8667504\,c_2^2\,\left[\mathsf{Ell}\right]^6-1474200 \,c_2\,\left[\mathsf{Ell}\right]^8+23328\, \left[\mathsf{Ell}\right]^{10}\,.$$

\noindent $\bullet$ For $i=2$ and $j=0$, we obtain 
$$398016 \,c_2^5\left[\mathsf{Ell}\right]-5472720\,c_2^4\,\left[\mathsf{Ell}\right]^3+11144448\,c_2^3\,\left[\mathsf{Ell}\right]^5-4477680\,c_2^2\,\left[\mathsf{Ell}\right]^7+279936\,c_2\,\left[\mathsf{Ell}\right]^9\,.$$

\noindent $\bullet$ For $i=0$ and $j=1$, we obtain
$$-3923136\,c_2^4\,\left[\mathsf{Ell}\right]^2+19223280\,c_2^3\,\left[\mathsf{Ell}\right]^4-17335008\,c_2^2\,\left[\mathsf{Ell}\right]^6+2948400\,c_2\,\left[\mathsf{Ell}\right]^8-46656\,\left[\mathsf{Ell}\right]^{10}\,.$$

\noindent $\bullet$ For $i=1$ and $j=1$, we obtain
$$-796032\,c_2^5\,\left[\mathsf{Ell}\right]+10945440\,c_2^4\,\left[\mathsf{Ell}\right]^3-22288896\,c_2^3\,\left[\mathsf{Ell}\right]^5+8955360\,c_2^2\,\left[\mathsf{Ell}\right]^7-559872\,c_2\,\left[\mathsf{Ell}\right]^9\,.$$

\noindent $\bullet$ For $i=2$ and $j=1$, we obtain
$$3923136\,c_2^5\,\left[\mathsf{Ell}\right]^2-19223280\,c_2^4\,\left[\mathsf{Ell}\right]^4+17335008\,c_2^3\,\left[\mathsf{Ell}\right]^6-2948400\,c_2^2\,\left[\mathsf{Ell}\right]^8+46656\,c_2\,\left[\mathsf{Ell}\right]^{10}\,.$$

\noindent These terms are to be {\it subtracted} from the terms in Remark \ref{r15} to yield the additional $6$ relations in the Chow group of $\mathcal F_2$ coming from the removal of $\widetilde{\mathsf{CTP}}.$
\end{rem}
\begin{rem} We can compare the calculations here with the results of \cite {CK}. In \cite[Theorem 1.2]{CK}, the Chow ring of the moduli space of elliptic $K3$ surfaces has been computed, and the ideal of relations has two generators in degrees $9$, $10$. An easy check shows that these relations arise by restricting the relations with $i=0$, $j=0$ and $i=1$, $j=0$ in Remark \ref{r15} and Remark \ref{r16} to the elliptic locus. 

Since the geometry of the weighted blowup and the geometry of $\mathsf{Ell}$ can be explicitly related,
the matching of relations is guaranteed. In fact, the relations in \cite[Remark 3.5]{CK} on $\mathsf{Ell}$ arise from removing classes supported on the codimension $9$ locus $\Delta$ of elliptic fibrations which are not stable, see \cite [Lemma 2.5(2)]{CK}. This corresponds precisely to the exceptional divisor of the blowup $\widetilde{\mathsf{CTP}}$, as discussed in Section \ref{TCdiscussion}.

\end{rem}
\subsection{Proof of Theorem \ref{main} {(iii)}}
    The quotient presentation of $\mathcal F_2$ in Section \ref{Shah} yields $$\mathsf A^{*}(\mathcal F_2)=\mathsf A^{*}_G(\tilde Y\smallsetminus (\widetilde{\mathsf{CTP}}\cup\widetilde {\mathsf{DC}}))\, .$$ In the ring $\mathsf A_G^{*}(X)=\qq[H,c_2,c_3]/(p)$, where\footnote{The exact expression for $p$ is easy to write down. However, this will not be needed for degree reasons.} \[p=H^{28}+c_1^G(\Sym^6 V^*)H^{27}+\cdots +c_{28}^G(\Sym^6 V^*)\, ,\] we quotient by the ideal generated by the relations from Lemmas \ref{nonreduced} and \ref{quadrelations} (explicitly written in Remark \ref{nonreducedremark} and Remark \ref{quadremark}) to obtain $\mathsf A^*_G(Y)$. Next, we carry out the weighted blowup along $\mathsf{TC}$, and we impose the relations in Corollary \ref{Ytilde} to obtain $\mathsf{A}_G^*(\tilde Y)$. The additional generator $\left[\mathsf{Ell}\right]$ appears at this stage. Finally, we impose the $6$ relations given by Corollary \ref{modulegenerators} and made explicit in Remark \ref{r15} and Remark \ref{r16}. In total, there are $14$ relations that we need to account for. Using the Macaulay2 package Schubert2 \cite{M2,S2} we find that 
\begin{align*}\sum_{k=0}^{19} t^k \cdot \dim \mathsf A_G^k(\tilde{Y}\smallsetminus \tilde{\mathsf{CTP}})=1 + &2t + 3t^2 + 5t^3 + 6t^4 + 8t^5 + 10t^6 + 12t^7 + 13t^8 + 14t^9 + 12t^{10} \\&+ 10t^{11} + 8t^{12} + 6t^{13} + 5t^{14} + 3t^{15} + 2t^{16} + t^{17}\,.\end{align*}

The above polynomial is precisely the polynomial in the statement of Theorem \ref{main}.
The space $\F_2$ is an open substack of $[(\tilde{Y}\smallsetminus \tilde{\mathsf{CTP}})/G]$ whose complement $\tilde{\mathsf{DC}}$ has components of codimension $17$ and $18$. By excision, it follows that the Poincar\'e polynomial of $\F_2$ agrees with that of $[(\tilde{Y}\smallsetminus \tilde{\mathsf{CTP}})/G]$, except for possibly the coefficient of $t^{17}$. We know, however, that $\dim \mathsf A^{17}(\F_2)\geq 1$ because $\lambda^{17}\neq 0$ \cite {kvg}. Therefore, we see that $\dim \mathsf A^{17}(\F_2)=1$.
\qed

\begin{rem}\label{soclerem} 
Let $I$ be the ideal of relations in Remarks \ref{nonreduced} and \ref{quadremark}. Using the Macaulay2 package Schubert2 \cite{M2,S2} we find that $$\mathsf A^*_G(X)/I=\mathsf A_G^{*}(Y\smallsetminus \mathsf{CTP})$$ has the Poincar\'e polynomial 
\begin{align}\label{YminusCTP}
\sum_{k=0}^{19}t^k\cdot \dim \mathsf A_G^k(Y\smallsetminus \mathsf{CTP})=1+&t+2t^2+3t^3+4t^4+5t^5+7t^6+8t^7+9t^8+9t^9+8t^{10}\\&+ 6t^{11}+5t^{12}+3t^{13}+3t^{14}+t^{15}+t^{16}\,.
\nonumber
\end{align}

By excision (and the fact that we can ignore the stratum $\widetilde{\mathsf{DC}}$), we have \begin{equation}\label{exseq}\mathsf A^{k-1}(\mathsf {Ell})\, \rightarrow\, \mathsf A^{k}(\mathcal F_2)\rightarrow {\mathsf A}^k_G(Y\smallsetminus \mathsf{CTP})\rightarrow 0\,.\end{equation}
The Chow Betti numbers of the right hand side are computed in \eqref {YminusCTP}, the Chow Betti numbers of the middle term are given in Theorem \ref{main} (iii), while the Chow Betti numbers of $\mathsf{Ell}$ are found in \cite [Theorem 1.2 (2)]{CK}:
\begin{align*}1 + t + 2 t^2 + 2 t^3 + 3 t^4 + 3 t^5 + &4 t^6 + 4 t^7 + 5 t^8 + 
 4 t^9 + 4 t^{10} + 3 t^{11} + 3 t^{12} + 2 t^{13} + 2 t^{14} + t^{15} + t^{16}.\end{align*}
Comparing ranks, it follows that the excision sequence \eqref{exseq} is exact to the left. 

In particular, the inclusion $r:\mathsf {Ell}\to \mathcal F_2$ induces an isomorphism between the top nonvanishing Chow groups $$r_*:\mathsf A^{16}(\mathsf {Ell})\, \stackrel{\sim}{\longrightarrow}\, \mathsf A^{17}(\mathcal F_2)\,.$$ We will use the isomorphism in Section \ref{cymap}.
\end{rem}

\begin{rem}\label{hc2} The description of the Chow ring of $\mathcal F_2$ in terms of generators and relations also allows us study the pairing $$\mathsf{R}^k(\mathcal F_2)\times \mathsf R^{17-k}(\mathcal F_2)\to \mathsf{R}^{17}(\mathcal F_2)\cong\mathbb Q.$$ For $k$
not equal to 8 or 9, the pairing is perfect.
The kernel of the paring 
$$\mathsf{R}^8(\mathcal F_2)\times \mathsf{R}^9(\mathcal F_2)\to \mathsf R^{17}(\mathcal{F}_2)
\cong\mathbb Q
$$
has rank $1$. The computations were carried out using code available at \cite{COP}. The class  in $\mathsf{A}^9(\mathcal{F}_2)$ generating the Gorenstein kernel is 
\begin{align*}
24&477984\, c_2^4\, \left[\mathsf{Ell}\right] - 336572280\, c_2^3\, \left[\mathsf{Ell}\right]^3 + 685383552\, c_2^2\, \left[\mathsf{Ell}\right]^5 
 -275377320\, c_2\,\left[\mathsf{Ell}\right]^7 + 17216064\, \left[\mathsf{Ell}\right]^9 \\&- 14925600 \,c_2^4\, H + 
 151121700\, c_2^2\, c_3\, H^2 + 108694300\, c_2^3\, H^3 - 
 125188470\, c_3^2\, H^3\\&- 305321805\, c_2\, c_3\, H^4 - 
 107744175\, c_2^2 H^5 + 65952495 c_3 H^6 + 19658950\, c_2\, H^7 - 
 646575\, H^9\,.
 \end{align*}
Finding a geometric interpretation of the kernel class is an open question. 
\end{rem}

\section {The cycle map} \label{cymap}
We present here the proof of part (iv) of Theorem \ref{main}. Throughout Section \ref{cymap},  $H_*$ will denote rational Borel--Moore homology \cite {BM}. In general, for any scheme or Deligne-Mumford stack $M$, the group $H_k(M)$ carries a mixed Hodge structure and an increasing weight filtration with weights between $-k$ and $0$. The cycle map takes values in the lowest weight piece of the Hodge structure  $$\mathsf{cl}: \mathsf {A}_k(M)\to W_{-2k}H_{2k}(M)\,.$$ If $M$ is nonsingular, we identify cohomology and Borel--Moore homology, but singular spaces will also enter the discussion.


We seek to show that the cycle map $$\mathsf{cl}:\mathsf A_k(\mathcal F_2)\to H_{2k} (\mathcal F_2)$$ is an isomorphism. Using the expressions for the Poincar\'e polynomial calculated in \cite{KL2} and Appendix \ref{appendix} together with the Chow Betti numbers from Theorem \ref{main} (iii), it suffices to prove that the cycle map is injective. 
We will prove below the following related injectivity.

\begin{lem}\label{lcycle} The cycle map $$\mathsf{cl}: {\mathsf A}_k^G(Y\smallsetminus \mathsf{CTP})\to W_{-2k}H_{2k}^{G}(Y\smallsetminus {\mathsf{CTP}})$$ is injective. 
\end{lem}

Assuming Lemma \ref{lcycle} for now, let $\alpha\in \mathsf A_k(\mathcal F_2)$ be so that $\mathsf{cl}(\alpha)=0$. We wish to show $\alpha=0.$ If $\alpha$ is not in the kernel of the intersection pairing in $\mathsf A^*(\mathcal F_2)$, we can find a class $\alpha'$ of complementary degree so that $\alpha\cdot \alpha'\neq 0.$ In particular, we may assume $\alpha\cdot \alpha'=\lambda^{17}$ since the latter generates $\mathsf A^{17}(\mathcal F_2).$ Then, $$0=\mathsf{cl}(\alpha)\cdot \mathsf{cl}(\alpha')=\mathsf{cl}(\lambda)^{17} .$$ However, the same argument used in Chow in \cite {kvg} shows that in cohomology we also have $\lambda^{17}\neq 0,$ yielding a contradiction. 

Thus $\alpha$ must be in the kernel of the intersection paring. In particular $\alpha$ must be the class in Remark \ref{hc2}, up to a multiple, but we will not use this fact in the proof below; we only care that $k>2$. We can write $$\alpha=\zeta^{*} \beta\,.$$ Here, we recall that $\zeta:\mathcal F_2\hookrightarrow[\widetilde Y/G]$ with complement $\widetilde{\mathsf{DC}}$ and $\widetilde{\mathsf{CTP}},$ and $\beta$ is a Chow class on $[\widetilde Y/G].$

Consider the diagram
\[
\begin{tikzcd}
\mathsf A^G_k(\widetilde{\mathsf{CTP}})\arrow[r, "\eta_*"]\arrow[d, "\mathsf{cl}"]& \mathsf A_k^G(\tilde Y) \arrow[r, "\zeta^*"]\arrow[d, "\mathsf{cl}"] & \mathsf A^G_k(\tilde Y\smallsetminus (\widetilde{\mathsf {CTP}}\cup\widetilde{\mathsf{DC}}))\arrow [r] \arrow[d, "\mathsf{cl}"]& 0                              \\
W_{-2k}H^G_{2k}(\widetilde{\mathsf{CTP}})\arrow[r, "\eta_*"]& W_{-2k}H_{2k}^G(\tilde Y) \arrow[r, "\zeta^{*}"] & W_{-2k}H^G_{2k}(\tilde Y\smallsetminus (\widetilde{\mathsf {CTP}}\cup \widetilde{\mathsf{DC}}))\arrow [r] & 0 \,.                            
\end{tikzcd}
\]
For the second excision sequence, exactness to the right follows since we keep track of the Hodge weights. For both exact sequences, we may ignore the two dimensional set $\widetilde{\mathsf{DC}}$ on the left terms for dimension reasons. 
Since $$\zeta^{*}\, (\mathsf{cl}(\beta))=\mathsf{cl}\, (\zeta^{*} 
(\beta))=\mathsf{cl} (\alpha)=0\,,$$ it follows that over $\tilde Y$, we have $$\mathsf{cl}(\beta)=\eta_{*}\,(\gamma)$$ where $\gamma$ is an equivariant Borel--Moore homology class on the locus $\widetilde{\mathsf{CTP}}.$ Using the blowdown map $f:\tilde Y\to Y,$ we obtain $$f_{*} \,\mathsf{cl} ({\beta})=f_{*}\, \eta_{*}\,(\gamma)\,,$$ where the right hand side is a Borel--Moore class on $\mathsf{CTP}.$ The restriction $f_{*}\, \mathsf{cl}(\beta)=\mathsf{cl}\,(f_{*}\beta)$ thus vanishes in the Borel--Moore homology of $\left[\left(Y\smallsetminus {\mathsf{CTP}}\right)/G\right]$, so by Lemma \ref{lcycle}, we conclude $$f_* \,(\beta)=0$$ in $\mathsf A^G_k(Y\smallsetminus \mathsf{CTP}).$ By excision, we can find a class $\delta$ such that on $Y$ we have $$f_{*}\, (\beta)=\bar \eta_{*} \,(\delta)\, ,$$ where $\bar \eta: \mathsf{CTP}\cap Y\hookrightarrow Y.$ In particular $$f_{*} (\beta - f^{*} \bar \eta_{*}\,(\delta))=0$$ in $\mathsf A^G_k(Y)$, hence $$\beta-f^{*} \bar \eta_{*}\, (\delta)$$ is a class supported on the exceptional divisor of the blowup $f:\tilde Y\to Y$ by excision applied to the embedding of the exceptional divisor in $\tilde Y$. We restrict the class $\beta-f^*\bar \eta_*\, (\delta)$ to $\mathcal F_2$ via $\zeta$, and note that $f^{*} \bar \eta_{*}\, (\delta)$ restricts trivially since we removed the strict transform $\widetilde {\mathsf{CTP}}.$
We conclude that $$\alpha=r_*\, (\epsilon)\,,$$ for a class $\epsilon$ on $\mathsf{Ell},$ where as usual $r:\mathsf {Ell}\to \mathcal F_2$ denotes the inclusion. We claim however that in this case $\alpha$ cannot be in the kernel of the intersection pairing unless $\alpha=0$. 

To see this last statement, recall from \cite{CK} that $\mathsf A^*(\mathsf{Ell})$ is Gorenstein. If $\epsilon\neq 0,$ we can find a complementary class $\epsilon'$ with $$\epsilon\cdot \epsilon'=\lambda^{16}\,.$$ The pullback $$r^{*}:\mathsf A^{*}(\mathcal F_2)\to \mathsf A^{*}(\mathsf{Ell})$$ is surjective, since two of the ring generators $\lambda$, $c_2$ on the left hand side are sent to the ring generators $\lambda$, $4c_2$ on the right hand side. Thus, we may write $$\epsilon'=r^{*}\xi\,.$$ Since $\alpha$ is in the kernel of the pairing, we have $$0=\alpha\cdot \xi=r_{*}\,(\epsilon)\cdot \xi = r_{*}\,(\epsilon\cdot r^{*}\xi)=r_{*}\,(\epsilon\cdot \epsilon')=r_{*}\,(\lambda^{16})\,.$$ This contradicts Remark \ref{soclerem}. \qed 
\vskip.1in
\noindent {\it Proof of Lemma \ref{lcycle}.} For simplicity, write $$Z=\mathsf {ML}\cup \mathsf{QP}\cup \mathsf{CTP}\subset X$$ where as before $X$ denotes the projective space of sextics. Then $Y\smallsetminus \mathsf{CTP}=X\smallsetminus Z.$ We need to establish the injectivity of the map $$\mathsf{cl}:{\mathsf A}^G_k(X\smallsetminus Z)\to W_{-2k}H_{2k}^G(X\smallsetminus Z)\,.$$ Consider the following excision diagram

\begin{equation}\label{cycleexcision}
\begin{tikzcd}
\mathsf A^G_k(Z)\arrow[r]\arrow[d, "\mathsf{cl}"]& \mathsf A_k^G(X) \arrow[r]\arrow[d, "\mathsf{cl}"] & \mathsf A^G_k(X\smallsetminus Z)\arrow [r] \arrow[d, "\mathsf{cl}"]& 0                              \\
W_{-2k}H^G_{2k}(Z)\arrow[r]& W_{-2k}H_{2k}^G(X) \arrow[r] & W_{-2k}H^G_{2k}(X\smallsetminus Z)\arrow [r] & 0 \,.                             \\
\end{tikzcd}
\end{equation}

We first claim that the middle cycle map is an isomorphism. Indeed, recall that $G=\SL(V)$ and let $K=\GL(V)$. We have an isomorphism $$\mathsf{A}^K_k(X)\to H_{2k}^{K}(X)\,.$$ This follows by explicitly computing both sides. In fact, both sides agree with the cohomology of the bundle $$X_K=\mathbb P(\Sym^6 E^{*})\to BK\, ,$$ where $E\to \BGL(V)$ is the universal bundle. To go further, we use the terminology of \cite [Section 4]{To}. There, two properties are singled out: the weak property is the statement that the cycle map is an isomorphism, while the strong property requires additional assumptions about odd cohomology, which vanishes for $X_K$. In other words, $\mathbb P(\Sym^6 E^{*})\to BK$ satisfies the strong property. To pass to the group $G$, we note that the mixed space $X_G\to X_K$ is a $\mathbb C^{*}$-bundle obtained from the total space of the determinant line bundle $F=\det \text{pr}^* \Sym^6 E^{*}$ on $\mathbb P(\Sym^6 E^{*})$ and removing the zero section. By homotopy equivalence, $F$ also satisfies the strong property since $X_K$ does, and the zero section satisfies it as well. The complement satisfies the weak property by \cite [Lemma 6]{To}, as claimed.

To show the rightmost cycle map is injective in the diagram \eqref{cycleexcision}, it suffices to show the leftmost cycle map is surjective. For simplicity, write $$Z_1=\mathsf{ML}\,,\quad Z_2=\mathsf{QP}\,,\quad Z_3=\mathsf{CTP}$$ for the three components of $Z$, and write $T_1, T_2, T_3$ for the nonsingular spaces that dominate them $$T_1=\mathbb P(V^{*})\times \mathbb P(\Sym^4 V^{*})\,,\quad T_2=\mathbb P K_{\textrm{quad}}\,, \quad T_3=\mathbb PK_{\textrm{ctp}}\, .$$ By Mayer--Vietoris in both Chow \cite [Example 1.3.1(c)]{Fulton} and Borel--Moore homology \cite [Theorem 3.10]{BM} and \cite[Theorem 5.35 and Remark 5.36]{PS}, we have a diagram 
\[
\begin{tikzcd}
\mathsf A^G_k(Z_1)\,\oplus\, \mathsf A^G_k(Z_2)\, \oplus \,\mathsf A^G_k(Z_3)\arrow[r]\arrow[d, "\mathsf{cl}"]& \mathsf A_k^G(Z) \arrow[r]\arrow[d, "\mathsf{cl}"] &0                      \\
W_{-2k}H^G_{2k}(Z_1)\,\oplus W_{-2k}H_{2k}^G(Z_2)\, \oplus W_{-2k}H^G_{2k}(Z_3)\arrow[r]& W_{-2k}H_{2k}^G(Z) \arrow[r] &0\,.           
\end{tikzcd}
\]
The surjectivity of the second row follows since the $(2k-1)$st Borel--Moore homology groups has no Hodge pieces of weight $-2k$. Therefore, to complete the proof we need to check the surjectivity of the cycle map on the left.

For each $1\leq i\leq 3,$ we form the diagram \[
\begin{tikzcd}
\mathsf A^G_k(T_i)\arrow[r]\arrow[d, "\mathsf{cl}"]& \mathsf A_k^G(Z_i) \arrow[r]\arrow[d, "\mathsf{cl}"] &0                      \\
W_{-2k}H^G_{2k}(T_i)\arrow[r]& W_{-2k}H_{2k}^G(Z_i) \arrow[r] &0\,.           
\end{tikzcd}
\] Surjectivity of the first row is standard (and, in fact, is not necessary for us), while surjectivity of the second row is found in \cite[Lemma A.4]{L} or \cite {P2}. 
The final step is
then to prove that the cycle map  on the left
is surjective. The
left cycle map is an isomorphism by the same argument used for $X$ using the explicit description of $T_1, T_2, T_3$ as iterated projective bundles over projective spaces. \qed


\appendix 
\section {The Poincar\'e polynomial of the moduli space}\label{appendix}

\subsection{The results of Kirwan and Lee} We discuss here the calculation of the Poincar\'e polynomial of $\mathcal F_2$ in \cite {KL1, KL2}. The value of the Poincar\'e polynomial given in \cite[Theorem 3.1] {KL2} is 
\begin{multline}\label{certainlyincorrect} 1 + 2q^2 + 3q^4 + 5q^6 + 6q^8 + 8q^{10} + 10q^{12} + 12q^{14} + 13q^{16} + 14q^{18} + 12q^{20} \\+ 10q^{22}  + 8q^{24} + 6q^{26} + q^{27} + 5q^{28} + 3q^{30} + q^{31} + 2q^{32} + q^{33} + 3q^{35}\, .\end{multline} 
However, the above  polynomial is incompatible with the geometry of the moduli space. Indeed, the projective Bailey-Borel compactification $$\mathcal F_2\hookrightarrow \overline{\mathcal F}^{\text{\,BB}}$$ has a $1$-dimensional boundary. Using this observation, it was shown in \cite {kvg} that \begin{equation}\label{nonvvv} \lambda^{17}\neq 0\in H^{34}(\mathcal F_2).\end{equation} In fact, intersecting two general hyperplane sections of $\overline{\mathcal F}^{\text{\,BB}}$ gives a compact $17$-dimensional subvariety of $\mathcal F_2$ on which $\lambda^{17}$ is non-zero by the ampleness of $\lambda$. However, this contradicts the vanishing $H^{34}(\mathcal F_2)=0$ implied by \eqref{certainlyincorrect}.

The value of the Poincar\'e polynomial used throughout our paper is 
\begin{multline}\label{hopefullycorrect} 1 + 2q^2 + 3q^4 + 5q^6 + 6q^8 + 8q^{10} + 10q^{12} + 12q^{14} + 13q^{16} + 14q^{18} + 12q^{20} \\+ 10q^{22}  + 8q^{24} + 6q^{26} + q^{27} + 5q^{28} + 3q^{30} + q^{31} + 2q^{32} + 2q^{33} + q^{34}+3q^{35}\, ,\end{multline} which differs from \eqref{certainlyincorrect} by $q^{33}+q^{34}$. The correction is aligned with the non-vanishing \eqref{nonvvv} of cohomology in degree $34.$

The main error in \cite{KL2} is in the proof, but not the statement, of Proposition 3.2. First, in \cite[Lemma 5.6]{KL2}, Kirwan and Lee claim to describe the image of a certain map $\tau_2^*$, but actually only describe a proper subspace of the image. This impacts the proof, but again not the statement, of \cite[Lemma 5.7]{KL2}. The inaccurate claim in the proof is used on \cite[page 581]{KL2} to study the kernel of another map $\chi^4$, ultimately leading to the erroneous Poincar\'e polynomial \eqref{certainlyincorrect}.

In Section \ref{correctedpoly}, we explain how to derive the correct Poincar\'e polynomial \eqref{hopefullycorrect} using the statement of \cite[Proposition 3.2]{KL2} together with the non-vanishing \eqref{nonvvv}. The latter fact was not used by Kirwan-Lee. In order to explain the issues regarding the proof of \cite[Proposition 3.2]{KL2}, a lengthier discussion of Kirwan-Lee's beautiful but intricate argument
is required. A correct derivation is explained in Section \ref{further} after we describe the geometric set-up and a few intermediate results in Sections \ref{KDesing} and \ref{Shahcompact}.  

\subsection{Kirwan's desingularization}\label{KDesing} The approach in \cite {KL1, KL2} starts with the GIT quotient of the space of sextics $$\overline{\mathcal F}^{\text{\,GIT}}=X\sslash G\, ,$$ where $X=\mathbb P^{27}$ and $G=\SL_3$. Kirwan's partial desingularization $$\overline{\mathcal F}^{\text{\,K}}=\widetilde X\sslash G$$ arises as a composition of four (weighted) blowups. It is obtained by first blowing up $X^{ss}$ along the orbits whose stabilizers have the highest dimension, deleting the unstable strata
in the blowup, and then repeating the same procedure to the resulting space. The partial desingularization $\overline{\mathcal F}^{\text{\,K}}$ possesses only finite quotient singularities, whereas the singularities of $\overline{\mathcal F}^{\text{\,GIT}}$ are more complicated. One of the main results of \cite {KL1} is the calculation of the Betti numbers of $\overline{\mathcal F}^{\text{\,K}}$:
\begin{multline}\label{kirwandesing}
1 + 5q^2 + 11 q^4 + 18 q^6 + 25 q^8 + 32 q^{10} + 40 q^{12} + 48 q^{14}
+ 55 q^{16} + 60 q^{18} + 60 q^{20} + 55 q^{22} \\+ 48 q^{24} + 40 q^{26} + 32q^{28}
+ 25 q^{30} + 18 q^{32} + 11q^{34} + 5q^{36} +q^{38}\,.
\end{multline}
This is used in \cite {KL2} to compute the Betti numbers of $\mathcal F_2$, viewing the latter as an open in $\overline{\mathcal F}^{\text{\,K}}.$\footnote{There are a few minor typos in the proof of \cite [Theorem 1.3]{KL1}. On the table in \cite [page 499]{KL1}, the locus labelled $(1, 0)$ corresponds to the stratum of unstable sextics of the form $\ell^5m$ where $\ell, m$ are distinct lines. This stratum contributes $$\frac{q^{46}}{(1-q^2)^2}\, .$$ 
As a result, formula \cite [Section 2.4, (1)]{KL1} should read
$$\frac{1-q^{50}}{(1-q^2)(1-q^4)(1-q^6)}-\frac{q^{20}-q^{28}}{(1-q^2)^3}\, .$$ Similarly, formula \cite[Section 4.2, (2)]{KL1} should be $$\frac{q^{22}-q^{42}}{(1-q^2)^2}\, .$$ Additionally, there is a misprint in \cite[Section 5.2, (1)]{KL1} which should read $$\frac{1+q^2}{(1-q^2)^2}(q^{16}+q^{32}-2q^{38})\, .$$ There are a few other small misprints but they do not affect the general argument.}

While the Chow groups of $\overline{\mathcal F}^{\text{\,K}}$ are not needed for our paper, in the spirit of Section \ref{cymap}, we expect that the cycle map $$\mathsf A^*(\overline{\mathcal F}^{\text{\,K}})\to H^{2*}(\overline{\mathcal F}^{\text{\,K}})$$ is an isomorphism.

\subsection{Shah's compactification}\label{Shahcompact} The first step of the desingularization procedure yields Shah's compactification $$\overline{\mathcal F}^{\text{\,Sh}}=X_1\sslash G\, .$$ Here $X_1$ is the weighted blowup of the triple conic locus $\mathsf{TC}$ in the locus of semistable sextics: $$\pi:X_1\to X^{ss}\, .$$ Indeed, $\mathsf{TC}$ is the orbit with the largest stabilizer, namely $R_0=\mathrm{SO}_3=\PSL_2$. 

Three further (unweighted) blowups are necessary to arrive at $\overline{\mathcal F}^{\text{\,K}},$ see \cite [page 504]{KL1}. The first of the remaining three blowups has as center the orbit $G \Delta$, where the reducible sextic $\Delta=(xyz)^2$ is invariant under the maximal torus $R_1$ in $G$. The final two blowups have as centers the orbits $G\widehat Z_{R_2}^{ss}$ and $G\widehat Z_{R_3}^{ss},$ where $\widehat Z_{R_2}^{ss},$ $\widehat Z_{R_3}^{ss}$ are the loci of semistable points fixed by two specific rank $1$ tori $R_2,$ $R_3$: $$R_2=\text{diag }\langle \lambda^{-2}, \lambda, \lambda\rangle, \quad R_3=\text{diag }\langle \lambda, \lambda^{-1}, 1\rangle, \quad \lambda\in \mathbb C^*\, .$$ In fact, the locus $G\widehat Z_{R_3}^{ss}$ (which will be relevant below) lies over the locus of products of three conics tangent at $2$ points \cite [Section 5.3]{KL1}. 

There is a blowdown map $$\overline{\mathcal F}^{\text{\,K}}\to \overline {\mathcal F}^{\text{\,Sh}}\, ,$$ which is an isomorphism over the stable locus $X_1^s/G$ of the Shah space. The moduli space $$\mathcal F_2\hookrightarrow X_1^s/G$$ is obtained by removing the union of a line and a surface $$Z=(A\smallsetminus \Delta)\cup (B\smallsetminus (B\cap D))\, .$$ The locus $A$ is a projective line passing through the point $\Delta=(xyz)^2$ and corresponds to the double cubic locus in $X$. The surface $B$ corresponds to the double conic + conic locus, and $B\cap D$ is a curve in $B$. 

The Poincar\'e polynomial of $X_1^s/G$ is computed in \cite [Proposition 3.2]{KL2}: 
\begin{multline}\label{x1s}
1 + 2 q^2 + 3 q^4 + 5 q^6 + 6 q^8 + 8q^{10} + 10 q^{12}
+ 12 q^{14} + 13 q^{16} + 14 q^{18} + 12 q^{20} \\+ 10 q^{22} + 8 q^{24} + 6 q^{26}
+q^{27} + 5q^{28} + 3 q^{30} +q^{31} + 2q^{32} +q^{33} +q^{34} +q^{35}.\end{multline}
This calculation is very important for the overall argument.\footnote{The intersection homology of the Shah compactification was computed in \cite [Theorem 1.2]{KL1}. On general grounds, see \cite [Remark 3.4]{invent}, the Betti numbers of $X_1^s/G$ agree with the intersection homology Betti numbers in degree less than roughly the dimension (up to a correction dictated by the unstable strata). In our case, this confirms the Poincar\'e polynomial of $X_1^s/G$ in degrees $\leq 16$. However, the remaining Betti numbers cannot be immediately derived from \cite{KL1}.}

\subsection {The Poincar\'e polynomial of $\mathcal F_2$}\label{correctedpoly} We confirm the Poincar\'e polynomial \eqref{hopefullycorrect} relying on equation \eqref{x1s}. Just as on \cite [page 580]{KL2}, we use the relative homology sequence for the pair $(X_1^s/G, \mathcal F_2)$. The difference is that we take into account that $H^{34}(\mathcal F_2)\neq 0$, thus leading to the different result \eqref{hopefullycorrect}. 

First, we note the Gysin isomorphism $$H_i(X_1^s/G, \mathcal F_2)=H^{38-i}_c(Z)\, .$$ In fact, we have $$H_c^0(Z)=0\,, \quad H_c^2(Z)=\mathbb Q\oplus \mathbb Q\,, \quad H_c^4(Z)=\mathbb Q\, ,$$ see \cite [(6.3)]{KL2}.
The relative homology sequence 
$$\ldots \to H_i(\mathcal F_2)\to H_i (X_1^s/G)\to H_c^{38-i}(Z)\to H_{i-1}(\mathcal F_2)\to\ldots$$
immediately yields isomorphisms $$H_i(\mathcal F_2)=H_{i}(X_1^s/G), \quad i\leq 32, \quad i=37, \quad i=38\, .$$ Expressions \eqref{certainlyincorrect}, \eqref{hopefullycorrect}, \eqref{x1s} all agree in degrees $\leq 32$ and $i=37$, $i=38$.  Furthermore, 
$$0\to H_{36} (\mathcal F_2) \to H_{36} (X_1^s/G) \to \mathbb Q \oplus \mathbb Q \to H_{35} (\mathcal F_2)\to  H_{35} (X_1^s/G)\to 0\, .$$ Using \eqref{x1s}, we have $$H_{36}(X_1^s/G)=0\,, \quad H_{35}(X_1^s/G)=\mathbb Q\, .$$ Therefore $$H_{36}(\mathcal F_2)=0 \,, \quad H_{35}(\mathcal F_2)=\mathbb Q\oplus \mathbb Q\oplus \mathbb Q\, ,$$ also in agreement with both \eqref{certainlyincorrect} and \eqref{hopefullycorrect}.

However, discrepancies appear in degrees $33$ and $34$. We have 
$$0 \to H_{34} (\mathcal F_2) \to H_{34} (X_1^s/G) \to \mathbb Q \to H_{33}(\mathcal F_2) \to H_{33}(X_1^s/G) \to 0\, .$$ Using $H_{34} (X_1^s/G)= \mathbb Q$ from \eqref{x1s} and the fact that $H_{34} (\mathcal F_2)\neq 0$ as noted above, it follows that the first map must be an isomorphism and $$H_{34}(\mathcal F_2)=\mathbb Q\, .$$ Using \eqref{x1s} one more time, we have $H_{33}(X_1^s/G)= \mathbb Q$, hence $$H_{33}(\mathcal F_2)=\mathbb Q\oplus \mathbb Q\, .$$ This confirms equation  \eqref{hopefullycorrect}. 

    

\subsection{Further discussion}\label{further} To give further credence to \eqref{hopefullycorrect}, we 
also identify the faulty reasoning in \cite {KL2}. To this end, we need to zoom in on the argument. We will explain that while \eqref{x1s} records the correct Poincar\'e polynomial of $X_1^s/G$, there are some errors in the derivation. The necessary corrections impact the last page of \cite {KL2}, and thus the final result. 

The strategy used to establish \eqref{x1s} is as follows: 
\begin{itemize} 
\item [(i)] a lower bound on the Betti numbers of $X_1^s/G$ is obtained from the Poincar\'e polynomial of $\overline{\mathcal F}^{\text{\,K}}$ in \eqref{kirwandesing} together with the relative homology sequence for the pair $$X_1^s/G\hookrightarrow \overline{\mathcal F}^{\text{\,K}}\, .$$ 
The resulting lower bounds for Betti numbers of $X_1^s/G$ are recorded in \cite [(3.11), (3.12)]{KL2}. 
\item [(ii)] Matching upper bounds are obtained in \cite [Sections 4, 5]{KL2}. The outcome is \cite[Corollary 5.11]{KL2}. 
\end{itemize}
In fact, steps (i) and (ii) are only carried out in degrees less or equal than $23$, while the higher terms are determined in \cite[Section 6]{KL2}.   

Step (i) requires the calculation of the Poincar\'e polynomial of the complement\footnote{The top term of the Poincar\'e polynomial of $Q$ in \cite[(3.8)]{KL2} should be $3q^{36}$, taking into account the correction $3q^{30}$ versus $3q^{20}$ in \cite[(3.5)]{KL2} and using the correct sign for the contribution of $E_{T, 2}\sslash G$. Similarly, there is a misprint in the first formula in \cite[page 569]{KL2} which requires the coefficient $6$ for $q^{26}$.} \cite [Section 3]{KL2}: $$Q=\overline{\mathcal F}^{\text{\,K}}\smallsetminus X_1^s/G\, .$$ Step (ii) examines the kernel of the restriction map $$\chi^*: H^*(\overline{\mathcal F}^{\text{\,K}})\to H^*(Q)\, .$$ This kernel is identified with the kernel of the restriction $$\rho^*: H^*(\overline{\mathcal F}^{\text{\,K}})\to H^*(\widehat E_1\sslash G) \oplus H^*(E_2 \sslash G) \oplus H^*(E_3 \sslash G)\, .$$ Here, $\widehat E_1\sslash G$ is the strict transform of the exceptional divisor $E_1 \sslash G$ of the second blowup (at $\Delta$), and $E_2\sslash G$ and $E_3 \sslash G$ are the exceptional divisors of the third and fourth blowup \cite [page 567]{KL2}. The correct identification of the kernel in codimension $4$ is needed in \cite [Section 6]{KL2} to determine the Betti numbers of $\mathcal F_2$ in high degrees. 

Before reviewing the analysis of the kernel of $\rho^*$ in \cite[Sections 4, 5] {KL2}, we need a few standard preliminaries. Consider the general setting of a $G$-equivariant blowup $$p:\tilde M \to M$$ of a nonsingular quasiprojective $M$ along a nonsingular equivariant center $N$ of codimension $c$, with exceptional divisor $E$. Note that the natural sequence \begin{equation}\label{blowupseq}0\to H_G^*(M)\to H_G^*(\tilde M)\to H_G^*(E)/H_G^*(N)\to 0\end{equation} induces an additive identification \cite [page 505]{KL1}: \begin{equation}\label{moreblowup}H_G^*(\tilde M)=p^* H_G^*(M) \oplus H_G^*(E)/H_G^*(N).\end{equation} Furthermore, $H_G^*(E)/H_G^*(N)$ has the additive basis 
$$\zeta^k \cdot p^* \alpha, \quad 1\leq k\leq c-1 \, ,$$
for classes $\alpha$ giving a basis of $H_G^*(N)$, and $\zeta$ denoting the hyperplane class of the projective bundle $E\to N$. For \eqref{moreblowup}, the splitting  $$H_G^*(E) / H_G^*(N)\to H_G^*(\tilde M)$$ of the natural restriction map in $\eqref{blowupseq}$ is not explicitly stated in \cite {KL1}. However, a splitting can be specified on the additive basis 
\begin{equation}\label{splitting}\zeta^k \cdot p^* \alpha \mapsto E^{k-1} \cdot j_! (p^* \alpha), \quad 1\leq k\leq c-1,\end{equation} with $j_!$ denoting the Gysin map for the closed immersion $E\to \tilde M.$ This convention is standard and is used for instance in \cite [page 495]{invent}. 

As an approximation of $\rho,$ one constructs spaces dominating the cohomology groups of the domain and target of $\rho$. For the domain, the space $\overline{\mathcal F}^{\text{\,K}}$ arises as a $4$-step blowup, and each blowup contributes to cohomology via \eqref{moreblowup}. Thus, the cohomology of $\overline{\mathcal F}^{\text{\,K}}$ has $5$ natural pieces, yielding generators\footnote{This uses Kirwan surjectivity \cite{K}; we only obtain generators after the unstable loci are deleted.}
\begin{equation}\label{ppp}p^*: F^*_1\oplus F^*_2\oplus F^*_3\oplus F^*_4\oplus F^*_5\twoheadrightarrow H^*(\overline{\mathcal F}^{\text{\,K}}).\end{equation} Similarly, there is a surjection  \begin{equation}\label{qqq} q^*: G^*_1\oplus G^*_2\oplus G^*_3\twoheadrightarrow H^*(\widehat E_1\sslash G) \oplus H^*(E_2 \sslash G) \oplus H^*(E_3 \sslash G).\end{equation} The interested reader can consult \cite [Section 4]{KL2} for a more detailed discussion and notation. There is an induced map on generators $$\sigma^*:F^*_1\oplus F^*_2\oplus F^*_3\oplus F^*_4\oplus F^*_5\to G^*_1\oplus G^*_2\oplus G^*_3\, ,$$ which is an approximation of $\rho^*.$ The kernel of $\sigma^*$ is calculated first. 

By \cite [page 576]{KL2}, the kernel of $\sigma^*$ consists of pairs $$(a, b)\in F^*_1 \oplus F^*_2, \quad \tau_1^*(a) + \tau_2^*(b) = 0, \quad \tau_4^*(a) = 0, \quad \tau_6^*(a) = 0\, .$$
Here, $$F_1^* = H^*(X) \otimes H^* (\BSL_3), \quad F_2^* = \tilde H^* (\mathbb P^{21})  \otimes H^* (\BSO_3)\, .$$
The space $F^*_1$ is the equivariant cohomology of the space of plane sextics. Additively, $F^*_2$ can be identified with the equivariant cohomology of the exceptional divisor of the first blowup, modulo the equivariant cohomology of the center of the blowup. Indeed, the codimension of the triple conic orbit is $27 - 5 = 22$, and the normalizer is $\mathrm{SO}_3.$ The maps $$\tau_1^*:F_1^*\to G_3^*, \quad \tau_2^*:F_2^*\to G_3^*$$ are introduced in \cite [page 571]{KL2}, while $\tau_4^*, \tau_6^*$ are constructed in \cite[pages 572--573]{KL2}. 
The target of $\tau_1^*$ and $\tau_2^*$ is the equivariant cohomology of the last exceptional divisor $G_3^*=H_G^*(E_3)$. However, the discussion of \cite[(4.6)]{KL2} shows that these maps factor through the equivariant cohomology $\widetilde G_3^*=H^*_G(G\widehat Z_{R_3}^{ss})$ of the center of the last blowup, followed by pullback:  $$\tau_1^*:F_1^*\to \widetilde G_3^*, \quad \tau_2^*:F_2^*\to \widetilde G_3^*\, .$$
A key step is to show \begin{equation}\label{keystep}\tau_1^* (a) + \tau_2^* (b) = \tau_4^*(a)=\tau_6^*(a)=0\implies \tau_1^* (a) = \tau_2^* (b) =0,\end{equation} see \cite [page 576]{KL2}. In turn, this relies on \cite[Lemma 5.6]{KL2} which specifies the image of $\tau_2^*$. 
 
It is important to understand the map $\tau_2^*$. This map arises from blowing up the codimension $18$ orbit $G \widehat Z_{R_3}^{ss}$ after all the other blowups have been carried out. To explain the notation, $\widehat Z_{R_3}^{ss}$ consists of the semistable points fixed by the torus $$R_3=\text{diag } \langle \lambda, \lambda^{-1}, 1\rangle \subset G=\SL_3\, .$$ This blowup is described in \cite [Section 5.3]{KL1}. It is noted in \cite [page 570]{KL2} that the cohomology of the center of the blowup is $$\widetilde G_3^*=H^*_G (G \widehat Z_{R_3}^{ss}) \cong H_{N(R_3)}^*(\widehat Z_{R_3}^{ss}) \cong H^* (\widehat Z_{R_3}^{ss}\sslash N(R_3)) \otimes H^* (BN_2)\, .$$ Here, $N(R_3)$ is the normalizer of $R_3$ in $G$ inducing a residual action on $\widehat Z_{R_3}^{ss}$, and $N_2$ is the normalizer of the maximal torus in $\SL_2.$ The first isomorphism is a general fact which follows from \cite [Corollary 5.6]{annals}, while the second isomorphism is explained in \cite [Section 5.3]{KL1}.

 The map $\tau_2^*$ can be described as taking classes on the exceptional divisor of the first blowup, viewing them as classes on the first blowup under the construction \eqref{splitting}, then restricting to $\widehat Z_{R_3}^{ss}$, while switching from $G$-equivariance to $N(R_3)$-equivariance \cite[page 571]{KL2}. Now recall that the surface $\widehat Z_{R_3}^{ss}\sslash N(R_3)$ carries an exceptional divisor $\theta$ obtained by blowing up the triple conic $\delta$ in $Z_{R_3}^{ss}\sslash N(R_3)$, see \cite [page 576] {KL2}. Lemma 5.6 in \cite {KL2} states that $$\text{Im }\tau_2^*\subset \widetilde G_3^*=H^*(\widehat Z_{R_3}^{ss}\sslash N(R_3))\otimes H^*(BN_2)$$ equals $\{\theta\} \otimes H^*(BN_2).$ This  is incorrect, and it should be replaced by classes supported on $\theta$, not the class of $\theta$ itself. This comes from the fact that in \eqref{splitting} self-intersections of the exceptional divisor also arise; see also \eqref {tau2} below. As a result, the conclusion 
$$\tau_1^*(a) = \tau_4^*(a) = \tau_6^* (a) = \tau_2^*(b)=0$$ in \cite[page 576]{KL2} does not hold. 

However, the strategy of the argument is sound, and \cite [Lemma 5.10]{KL2} can be salvaged. We will give the details below. We first consider the equations $$\tau_4^* (a) = \tau_6^* (a) = 0\, ,$$ where we may assume $a$ has degree $\leq 23,$ since the lemma only concerns such degrees. 
By \cite[Lemma 5.4]{KL2}, the intersection $\text{Ker }\tau_4^* \cap \text{Ker }\tau_6^*$ in the polynomial ring $$F^*_1=H^*(\mathbb P^{27})\otimes H^*(\BSL_3)=\mathbb Q[H, c_2, c_3]/(H^{28})$$ is generated by two classes of degrees $4$ and $14$, namely $$H^2 \,\,\,\text{and}\,\,\, \alpha=H\cdot (4c_2^3+27c_3^2)\, .$$ We can explicitly list all classes in the intersection of the two kernels of $\tau_4^*$ and $\tau_6^*$. The ranks of $\text{Ker }\tau_4^*\cap \text{Ker }\tau_6^*$ are given in each degree by $$q^4 + q^6 + 2 q^8 + 3 q^{10} + 4 q^{12} + 6 q^{14} + 7 q^{16} + 9 q^{18} + 11 q^{20} + 13 q^{22}\, ,$$ in agreement with \cite[page 577]{KL2}.
For instance, in degree $4$ (the simplest case), we have the unique class $H^2$. In degree $20$ (the most involved case), we have the $11$ classes 
$$H^{10}, H^8 c_2, H^7 c_3, H^6 c_2^2, H^5 c_2 c_3, H^4 c_3^2, H^4 c_2^3, H^3 c_2^2 c_3, H^2 c_2^4, H^2 c_2 c_3^2, \alpha c_3\, ,$$ and the  classes $a$ of degree $20$ lie in the span of these terms. 
Similarly, $$b\in F^*_2=\tilde H^*(\mathbb P^{21}) \otimes H^*(\BSO_3)\, .$$ Let $\zeta$ be the hyperplane class on the first exceptional divisor. The ranks of $F^*_2$ in degrees $\leq 23$ are immediately calculated to be 
$$q^2 + q^4 + 2 q^6 + 2 q^8 + 3 q^{10} + 3 q^{12} + 4 q^{14} + 4 q^{16} + 5 q^{18} + 5 q^{20} + 6 q^{22}\, .$$ For example, in degree $4,$ we have the class $\zeta^2$. In 
degree $20$, we have the $5$ classes $$\zeta^{10}, \, \zeta^8 c_2, \, \zeta^6 c_2^2, \, \zeta^4 c_2^3, \,\zeta^2 c_2^4\, ,$$ and the  classes $b$ of degree $20$ lie in the span of these terms. 
By the proof of \cite [Lemma 5.4]{KL2}, the homomorphism $$\tau_1^*:F^*_1 \to \widetilde G^*_3=H^*(\widehat Z_{R_3}^{ss}\sslash N(R_3))\otimes H^*(BN_2)$$ is given by 
\begin{equation}\label{tau1}\tau_1^*(H) = C\otimes 1, \quad \tau_1^* (c_2) = - 1\otimes \xi  + n ( [\text{pt}]\otimes 1), \quad \tau_1^* (c_3) =  C'\otimes \xi,\end{equation} where $n$ is an integer, $\xi$ is the degree $4$ generator of $H^*(BN_2)$, and $C, C'$ are curves on the surface $\widehat Z_{R_3}^{ss}\sslash N(R_3).$ In fact, by construction $C^2=1$. Similarly, \begin{equation}\label{tau2}\tau^*_2 (\zeta) = \theta\otimes 1,\quad \tau^*_2(c_2)=- 1\otimes \xi \implies \tau^*_2 (\zeta^2) = -  [\text{pt}]\otimes 1, \quad \tau^*_2 (\zeta^k) =0,  \quad k\geq 3.\end{equation} The last vanishing can be seen using the construction \eqref{splitting} and the fact that we restrict to a surface $\widehat Z_{R_3}^{ss}\sslash N(R_3).$ 
Now we can solve \eqref{keystep} in each degree using \eqref{tau1} and \eqref{tau2}. For instance, in degree $4$, we need 
$$a\cdot \tau^*_1 (H^2) + b \cdot \tau^*_2 (\zeta^2) =0\implies a-b=0$$ so the kernel is spanned by $H^2+\zeta^2$. In degree $20$, we write \begin{align*}a &= a_1 \cdot H^{10} + a_2\cdot H^8 c_2 +a_3 \cdot H^7 c_3 + a_4\cdot  H^6 c_2^2 + a_5 \cdot H^5 c_2 c_3 + a_6 \cdot H^4c_3^2 + a_7 \cdot H^4 c_2^3 \\ &\hskip.5in\hskip.5in\hskip.5in\hskip.5in\hskip.5in+a_8 \cdot H^3 c_2^2 c_3 + a_9 \cdot H^2 c_2^4 + a_{10} \cdot H^2 c_2 c_3^2 + a_{11}\cdot \alpha c_3\,,\\
b &= b_1 \cdot \zeta^{10} + b_2 \cdot \zeta^8 c_2 +b_3\cdot \zeta^6 c_2^2 +b_4 \cdot \zeta^4 c_2^3 +b_5 \cdot \zeta^2 c_2^4\,.\end{align*} From here, using \eqref{tau1}, \eqref{tau2}, we find 
$$\tau_1^*(a)+\tau_2^*(b)=0\iff a_9 - 4 a_{11} \cdot (C.C') - b_5 = 0\, .$$ This yields a $15$-dimensional solution space. 
After solving \eqref{keystep} in each degree, we find the ranks of $\text{Ker }\sigma^*$ in degree $\leq 23$ are accurately recorded by \cite[Lemma 5.7]{KL2}:  
$$q^4 + 2q^6 + 3q^8 + 5q^{10} + 6q^{12} + 8q^{14} + 10 q^{16} + 12 q^{18} + 15 q^{20} + 17 q^{22}\, .$$ However, the description of the kernel of $\sigma^*$ in degree $4$ needs to be corrected. As already mentioned, this impacts the argument in \cite [page 581]{KL2}. 

At this stage, thanks to \cite [Lemma 5.7]{KL2}, we have complete knowledge of the kernel of $\sigma^*$ in degree $\leq 23$. The next results \cite [5.9 - 5.11]{KL2} concern $\text{Ker }\rho^*$ which is required in part (ii) above. No correction to the statements in \cite{KL2} is needed here. However, the derivation of \cite [Lemma 5.10]{KL2}
crucially uses \cite [(5.1)]{KL2}. This derivation requires a few modifications to the values in \cite{KL2}. Up to order $23$, we have: 
\begin{align}\label{kern}
&\text{Ker } p^* = q^{16} + 5 q^{18} + 14 q^{20} + 28 q^{22}\,,\\ 
&\text{Ker }q^*_{11} = q^{18} + 3 q^{20} + 6 q^{22}\,,\nonumber\\
&\text{Ker } q^*_2 = q^{16} + 3 q^{18} + 5 q^{20} +8 q^{22}\,, \nonumber \end{align}
while
$$\text{Ker } q^*_3 = q^{18} + 5 q^{20} + 10 q^{22}$$
is correct in \cite {KL2}. Here, $p^*$ is introduced in \eqref {ppp}, and $q^*_{11}, q^*_2, q^*_3$ are certain components of the morphism \eqref{qqq}. Thus, using \cite[(5.9)]{KL2}, the expression 
$$\text{Ker }q^*_{11} + \text{Ker }q^*_2 + \text{Ker }q^*_3 + \text{Ker }\sigma^* - \text{Ker }p^*$$
yields the upper bound for $\text{Ker }\rho^*=\text{Ker }\chi^*$ in \cite[Lemma 5.10]{KL2} to be 
$$q^4 + 2 q^6 + 3 q^8 + 5 q^{10} + 6 q^{12} + 8 q^{14} + 10 q^{16} + 12 q^{18} + 14 q^{20} + 13q^{22}\, .$$ This completes step (ii), and also confirms \cite[Proposition 3.2]{KL2} and equation \eqref{x1s} along with it. 

The method of computing of the ranks of $\text{Ker }p^*$ and $\text{Ker }q^*_2$ is described in \cite[(5.1)]{KL2}, but the details are suppressed and the results are recorded imprecisely. For instance, $\text{Ker }p^*$ receives the following $6$ contributions:
\begin{itemize}
\item from the domain of $p^*$, the term $F^*_1 = H^* (\mathbb P^{27}) \otimes H^*(\BSL_3)$ contributes
$$\frac{1-q^{56}}{1- q^2}\cdot \frac{1}{(1-q^4)(1-q^6)};$$
\item next, $F^*_2= \tilde H^*(\mathbb P^{21}) \otimes H^*(\BSO_3)$ contributes
$$\frac{q^2- q^{44}}{1 - q^2}\cdot \frac{1}{1-q^4};$$ 
\item the remaining pieces of the domain of $p^*$ are found in \cite [page 571]{KL2}. We have $F_3^*=\tilde H^*(\mathbb P^{20}) \otimes H^*(BN)$ which contributes $$\frac{q^2 - q^{42}}{1-q^2}\cdot \frac{1}{(1-q^4)(1-q^6)}\, .$$ Here $N$ is the normalizer of the maximal torus $R_1$ in $G$ (this is denoted $N_3$ in \cite {KL2});
\item similarly for $F_4^*=H^*(\widehat Z_{R_2}\sslash N(R_2)) \otimes H^*(B\mathbb C^*) \otimes \tilde H^*(\mathbb P^{18})$ we get the contribution
$$(1 + q^2)\cdot \frac{1}{1-q^2}\cdot \frac{q^2 - q^{38}}{1-q^2}\, .$$
The first term is computed in \cite[Section 5.1]{KL1};
\item for $F_5^* = H^*(\widehat Z_{R_3}\sslash N(R_3)) \otimes H^*(BN_2) \otimes \tilde H^*(\mathbb P^{17})$ we get        
$$(1 + 3q^2 + q^4)\cdot \frac{1}{1-q^4} \cdot \frac{q^2-q^{36}}{1-q^2}\, .$$
The first term was computed in \cite[Section 5.3]{KL1};
\item for the target of $p^*$, the contribution of $H^*(\overline{\mathcal F}^{\text{\,K}})$ is recorded in \eqref{kirwandesing}. 
\end{itemize}
Putting these contributions together, we find that $\text{Ker }p^*$ is given by \eqref{kern}, as claimed. The discrepancy with the value in \cite [(5.1)] {KL2} is $3q^{22} \mod q^{24}$. 

Next, we examine $q_2^*: G_2^*\twoheadrightarrow H^*(E_2\sslash G).$ The dimension of the target is recorded in \cite[(3.5)]{KL2}: 
\begin{multline*}(1+q^2)(1+2q^2+3q^4+4q^6+5q^8+6q^{10}+7q^{12}+8q^{14}+8q^{16}+8q^{18}+8q^{20}+7q^{22}),
\end{multline*}
up to order $23$. By \cite [(4.3)]{KL2}, the domain is 
$$G_2^* = H^*(\widehat Z_{R_2}\sslash N(R_2)) \otimes H^*(B\mathbb C^*) \otimes H^*(\mathbb P^{18})\, ,$$ whose contribution equals $$(1+q^2)\cdot \frac{1}{1-q^2}\cdot \frac{1-q^{38}}{1-q^2}\, .$$  
Subtracting the two series above, we find the dimension of $\text{Ker } q^*_2$ matching the last equation in \eqref{kern}. The value recorded \cite [(5.1)]{KL2} is different. The misprint likely originates with \cite [Section 5.2]{KL1}. 

Finally, we consider the map $$q^*_1: G^*_1\to H^*(\widehat E_1\sslash G)$$ discussed in \cite[(4.4)]{KL2}.  Here, $\widehat E_1$ is a blowup of $E_1^{ss}$ described in \cite [page 570]{KL2}. As noted in \cite [page 578]{KL2}, the map $q^*_1$ has three components $q^*_{11},$ $q^*_{12},$ $q^*_{13}$, where $$q^*_{11}: H_G^*(E_1)\to H^*(\widehat E_1\sslash G)\, .$$ Since $\widehat E_1$ has no strictly semistable points, we have $$H^*(\widehat E_1\sslash G)=H_G^*(\widehat E_1^{ss})\, ,$$  see also \cite [page 570]{KL2}. We factor $q_{11}^*$ as the composition $$H^*_G(E_1)\stackrel{f^*}{\longrightarrow} H^*_G(E_1^{ss})\stackrel{g^*}{\longrightarrow} H^*_G(\widehat E_1)\stackrel{h^*}{\longrightarrow}  H^*_G(\widehat E_1^{ss})\, .$$ The first map $f^*$ is surjective on general grounds \cite{K}. We will compute its kernel below. The middle map $g^*$ is a pullback induced by a blowup so it is injective. The third map $h^*$ removes unstable strata from the blowup $\widehat E_1$ to arrive at $\widehat E_1^{ss}$. 

The calculation of the kernel of the surjection $f^*: H^*_G(E_1)\to H^*_G(E_1^{ss})$ is a matter of recording dimensions. 
\begin{itemize} 
\item [(a)] For the domain, we have $H^*_G (E_1) = H^* (BN) \otimes  H^*(\mathbb P^{20})$ \cite [page 570]{KL2}. The first factor comes from center of the blowup using the isomorphism $G\Delta=G/N$. We recall that $N$ stands for the normalizer of the maximal torus in $G$. The projective space $\Sigma=\mathbb P^{20}$ corresponds to the projectivization of the normal bundle at $\Delta$ of the orbit $G\Delta$ in $X$. This contributes $$\frac{1}{(1-q^4)(1-q^6)} \cdot \frac{1-q^{42}}{1-q^2}\, .$$ 
\item [(b)] For the target, on general grounds we have $$H^*_G(E_1^{ss}) = H^*_G(G\times_N (\mathbb P^{20})^{ss})=H^*_N((\mathbb P^{20})^{\text{ss}})\, .$$ The $N$-equivariant Poincar\'e series of $(\mathbb P^{20})^{ss}$ was calculated in \cite [Section 6.5]{KL1}, equation (1). Expanding up to order $23$ we find  
$$1 + q^2 + 2 q^4 + 3 q^6 + 4 q^8 + 5 q^{10} + 7 q^{12} + 8 q^{14} + 10 q^{16} + 11 q^{18} + 11 q^{20} + 10 q^{22}\, .$$  
\end{itemize} 
The kernel of $f^*$ up to order $23$ is determined from here by subtracting the two expressions (a) and (b). The answer reproduces the value claimed on the second line of \eqref{kern}. 

Since $g^*$ is injective, we have $$\text{Ker }g^*\circ f^*= \text{Ker }f^*\, .$$ We furthermore claim this agrees with the kernel of $q_{11}^*=h^*\circ g^*\circ f^*.$ To this end, it suffices to show \begin{equation}\label{inters}\text{Ker }h^* \cap \text{Im }g^* = 0 .\end{equation} Indeed, we need to rule out the situation that classes supported on unstable strata of $\widehat E_1$ might equal a class pulled back from $E_1^{ss}$. Should this happen, removing the unstable stratum from $\widehat E_1$ will also kill additional classes on $E_1^{ss}$, thus increasing the kernel of $q_{11}^*$ when compared to the kernel of $f^*$. The discussion might have been clear to the authors of \cite {KL1, KL2} and it was not recorded explicitly, but we indicate here a possible argument. 

Let $\Sigma=\mathbb P^{20}$. By the above remarks (a) and (b), the maps $f^*$, $g^*$ and $h^*$ can be rewritten as $$H_{N}^*(\Sigma)\stackrel{f^*}{\longrightarrow} H_{N}^*\left(\Sigma^{ss}\right)\stackrel{g^*}{\longrightarrow} H_{N}^*(\widehat \Sigma)\stackrel{h^*}{\longrightarrow} H_{N}^*(\widehat \Sigma^{ss})$$ where $\widehat \Sigma$ is the blowup of $\Sigma^{ss}$ along the $N$-orbits of the $R_2$-fixed and $R_3$-fixed loci. 
It is explained in \cite[Section 4.3]{KL1} that $$\Sigma=\mathbb P(W)=\mathbb P^{20}$$ where $W$ is the subspace of sextics spanned by the $21$ monomials $x^i y^j z^{6-i-j}$ for $i, j\geq 0$, $i+j\leq 6$ and $$(i, j)\not \in \{(2, 2), (2, 1), (2, 3), (3, 1), (3, 2), (1, 3), (1, 2)\}\, .$$ Using the terminology of \cite [page 508]{KL1}, these monomials are obtained from the ``Hilbert diagram" in \cite [page 497]{KL1} by removing the middle hexagon. 
\vskip.1in
Recalling $R_2=\text{diag }\langle \lambda^{-2}, \lambda, \lambda \rangle,$ it follows that the $R_2$-fixed locus is the projective line $$\widehat Z_{R_2}=\mathbb P \langle x^2y^4, x^2z^4\rangle\, .$$ Let $Q_2$ is the normalizer of $R_2$ in $N,$ which is easily computed to be isomorphic to the normalizer of the maximal torus in $\GL_2$. We have $Q_2/R_2=\mathbb C^*\rtimes \mathbb Z/2\mathbb Z$. It is easy to see that $R_2$ acts trivially on $\widehat Z_{R_2}^{ss}=\mathbb P^1$, and the $\mathbb C^*$-factor of $Q_2/R_2$ acts with equal opposite weights. 
On general grounds $\widehat Z_{R_2}^{ss}$ consists in the $Q_2$-semistable points of $\widehat Z_{R_2}$, see \cite [Remark 5.5]{annals}. It follows that the unstable points are $x^2y^4$ and $x^2z^4$, so $\widehat Z_{R_2}^{ss}=\mathbb C^*$ Thus, the equivariant cohomology of the orbit is $$H^*_{N}(N \widehat Z_{R_2}^{ss})=H^*_{Q_2}(\widehat Z_{R_2}^{ss})=H^*(BR_2)=H^*(B\mathbb C^*)\, .$$ This is in agreement with \cite [page 570]{KL2}. 

We consider the blowup of $\Sigma^{ss}$ along the orbit of $\widehat Z_{R_2}^{ss}.$ The exceptional divisor $F$ of the blowup is a $\mathbb P^{18}$-bundle over the base. 
We need to identify the unstable locus in the exceptional divisor. The weights of the representation of $R_2$ on $\mathbb P^{18}$ can be lifted from \cite [Section 5.2]{KL1}. Up to an overall factor of $-3$, they are $-2, -1, 1, 2, 3, 4$ with multiplicities $7, 4, 2, 3, 2, 1.$ Thus, in suitable coordinates, the action is given by $$\lambda \cdot [x:y:z:w:t:s]=[\lambda^{-2} x: \lambda^{-1}y: \lambda z: \lambda^2 w: \lambda^3 t: \lambda^4 s]\, ,$$ where $$(x, y, z, w, t, s)\in \mathbb C^7\oplus \mathbb C^4\oplus \mathbb C^2\oplus \mathbb C^3\oplus \mathbb C^2\oplus \mathbb C\, $$ are not all zero. The unstable locus is easily seen to be the union $$\mathbb P^{10}\sqcup \mathbb P^7\, ,$$ corresponding to $z=w=t=s=0$ and $x=y=0$ respectively. This conclusion is in agreement with the Poincar\'e polynomial calculation in \cite [Section 5.2 (1)]{KL1}. Letting $\epsilon$ denote the equivariant parameter for $R_2$, and letting $H$ denote the hyperplane class on $\mathbb P^{18}$, we compute the $R_2$-equivariant classes $$\big[\mathbb P^{10}\big]=(H+\epsilon)^2(H+2\epsilon)^3(H+3\epsilon)^2(H+4\epsilon), \quad \big[\mathbb P^{7}\big]=(H-2\epsilon)^7(H-\epsilon)^4\, .$$ Similar expressions hold for the classes of all equivariant linear subspaces of $\mathbb P^{10}$ or $\mathbb P^7$: the monomials in $H$ and $\epsilon$ above will have different exponents. By inspection, nonzero combinations of such classes never come from the base of the blowup $H^*(BR_2)=\mathbb Q[\epsilon]$ by pullback. This is the key to establishing \eqref{inters}. 

To this end, the reader may find the following diagram useful: 
\[
\begin{tikzcd}
& & H_{N}^{*}(\widehat S)\arrow[r, "\bar j^{*}"]\arrow[d, "i_!"] & H_N^*(\widehat S\cap F)\arrow[d, "\bar i_!"] \\ 
0\arrow[r] &H^*_N(\Sigma^{ss})\arrow[r, "g^*"]\arrow[rd]& H^*_N(\widehat \Sigma) \arrow[r, "j^*"]\arrow[d, "h^*"] & H^*_N(F)\,.                          \\
& & H_{N}^*(\widehat \Sigma^{ss})\,          
\end{tikzcd}
\]
Here, $\widehat S$ is the unstable locus of $\widehat \Sigma.$ The three-term column of the diagram is the Gysin sequence for the closed subvariety $\widehat S\subset \widehat \Sigma.$ For the first term, the cohomology is shifted by codimension, but the notation does not indicate this explicitly. In fact, $\widehat S$ is not pure dimensional, the individual connected components need to be considered separately. 

On general grounds \cite {K}, the unstable locus admits a stratification by locally closed nonsingular subvarieties 
 $$\widehat S=\bigsqcup_{\beta} \widehat S_\beta\, .$$ The intersection $$\widehat S\cap F=\bigsqcup_{\beta} (\widehat S_{\beta}\cap F)$$ is the union of unstable strata of the exceptional divisor $F$, see the proof of \cite [Proposition 7.4]{annals}. 
The inclusion $\bar j$ induces an isomorphism in cohomology $$\bar j^*:H_N^*(\widehat S)\to H_N^*(\widehat S\cap F)\, .$$ Indeed, it is shown in the proof of \cite [Proposition 7.4]{annals} that for each individual stratum, the inclusion induces an isomorphism $$\bar j_{\beta}^{*}:  H^*_N(\widehat S_\beta)\to H^*_N(\widehat S_\beta\cap F)\,.$$  Comparing the spectral sequence of the two stratifications of $\widehat S$ and $\widehat S\cap F$ (or equivalently by comparing the Gysin sequences induced by adding the unstable strata one at a time), we conclude the same is true about the map $\bar j^*.$ 

By the above discussion and \cite [Lemma 7.8]{annals}, we see that $$\widehat S\cap F\to N\widehat Z_{R_2}$$ is a $\mathbb P^{10}\sqcup \mathbb P^{7}$-fibration contained in the $\mathbb P^{18}$-fibration $F\to N\widehat Z_{R_2}.$ To establish \eqref{inters}, let $$\alpha\in H_N^*(\widehat \Sigma), \quad \alpha \in \text{Ker }h^* \cap \text{Im }g^*\,.$$ Then, from the third-term column of the diagram, we have $$\alpha=i_!(\gamma), \quad \gamma\in H^*_N(\widehat S)\,.$$ We compute $$j^*\alpha=j^*i_!(\gamma)=\bar i_! \bar j^*(\gamma)\,.$$ The class $j^*\alpha$ must come from the base of the $\mathbb P^{18}$-fibration $F\to N\widehat Z_{R_2},$ since $\alpha$ is in the image of $g^*$. However, the argument in the paragraphs above shows that classes $\bar j^*(\gamma)$ supported on the unstable part $\widehat S\cap F$ do not come from the base, unless of course $$\bar j^*(\gamma)=0\,.$$ Using that $\bar j^*$ is an isomorphism, we must have $\gamma=0$, hence $\alpha=i_!(\gamma)=0$ as claimed by \eqref{inters}. 

A similar analysis applies to $R_3=\text{diag }\langle \lambda, \lambda^{-1}, 1\rangle$, so $$\widehat Z_{R_3}=\mathbb P \langle x^3y^3, xy z^4, z^6\rangle \, .$$ The blowup is a $\mathbb P^{17}$-bundle over the base, and the representation of $R_3$ is computed in \cite [Section 5.4]{KL1}. The unstable locus is similarly a projective bundle over the base. An analogous argument applies in this case as well.

\end{document}